\documentclass[11pt]{amsart}
\usepackage{amssymb,amsthm,amsmath,amsfonts}

\usepackage[utf8]{inputenc}
\usepackage{geometry}
\usepackage{amssymb}
\usepackage{verbatim}
\usepackage{dsfont}
\usepackage{enumerate}
\usepackage{verbatim}
\usepackage{amsfonts}
\usepackage{graphicx}
\usepackage{tikz}
\usetikzlibrary{matrix}
\def\ocirc#1{\ifmmode\setbox0=\hbox{$#1$}\dimen0=\ht0
    \advance\dimen0 by1pt\rlap{\hbox to\wd0{\hss\raise\dimen0
    \hbox{\hskip.2em$\scriptscriptstyle\circ$}\hss}}#1\else
    {\accent"17 #1}\fi}

\newcommand{\R}{\mathbb{{R}}}
\newcommand{\FPh}{\mathcal{F}^\Phi}
\newcommand{\rn}{{\mathbb{R}^n}}
\newcommand{\mup}{{\mu_{\Phi}}}

\newcommand{\ve}{{\varepsilon}}
\newcommand{\dv}{{\rm div\,}}
\newcommand{\xinm}{{\xi_n^{\rm meas}}}
\newcommand{\gam}{{\gamma^{\rm meas}}}
\newcommand{\etam}{{\eta^{\rm meas}}}
\newcommand{\Mb}{{\mathcal{M}_b}}
\newcommand{\B}{{\mathcal{B}}}
\newcommand{\MP}{{\mathcal{M}^\Phi_b}}
\newcommand{\capP}{{\mathrm{C}_\Phi}}
\newcommand{\CapP}{{\mathrm{cap}_\Phi}}

\newcommand{\e}{\varepsilon}
\newcommand{\opA}{{\mathcal{ A}}}
\newcommand{\vp}{\varphi}
\newcommand{\dd}{\mathrm{d}}

\usepackage{color}

\newcommand{\N}{\mathbb{N}}
\newcommand{\wt}{\widetilde}

\theoremstyle{definition}

\def\rp0{{[0,\infty)}}

\newtheorem{theorem}{\bf Theorem} 
\newtheorem{corollary}{\bf Corollary}[section]
\newtheorem{lemma}[corollary]{\bf Lemma}
\newtheorem{remark}[corollary]{\bf Remark} 
\newtheorem{definition}[corollary]{\bf Definition} 
 
\newtheorem{fact}[corollary]{\bf Fact} 
\newtheorem{proposition}[corollary]{\bf Proposition}
\newtheorem{example}{Example} 
\swapnumbers

\title{Essentially fully anisotropic Orlicz functions\\ and uniqueness to measure data problem}
\author{Iwona Chlebicka and Piotr Nayar}\address{Iwona Chlebicka and Piotr Nayar\\Faculty of Mathematics, Informatics and Mechanics, University of Warsaw\\ul. Banacha 2, 02-097 Warsaw, Poland} \email{\texttt{i.chlebicka@mimuw.edu.pl,\ p.nayar@mimuw.edu.pl}}
\date{May 2019}
\begin{document}

\subjclass[2010]{46E30, (46E35)\vspace{1mm}} 

\keywords{Anisotropy, Capacity, Lebesgue points, Measures, Orlicz--Sobolev spaces\vspace{1mm}}

\thanks{{\it Acknowledgements.}\  I. C. acknowledges the support by NCN grant no. 2016/23/D/ST1/01072. 
\vspace{1mm}}

 \pagestyle{empty}
 \begin{abstract} Studying elliptic measure data problem with strongly nonlinear operator whose growth is described by the means of fully anisotropic $N$-function, we prove the uniqueness for a broad class of measures. In order to provide it, the framework of capacities in fully anisotropic Orlicz-Sobolev spaces is developed and the~capacitary characterization of a~bounded measure is given. 
 
 Moreover, we give an example of an anisotropic Young  function $\Phi$, such that $|\xi|^p \lesssim\Phi(\xi)\lesssim |\xi|^p\log^\alpha(1+|\xi|)$, with arbitrary $p\geq 1$, $\alpha>0$, but so irregularly growing that 
  the Orlicz--Sobolev--type space generated by $\Phi$ indispensably requires fully anisotropic tools to be handled.
 \end{abstract}

\maketitle


\section{Introduction}

Our aim is twofold -- to provide a method of construction of essentially anisotropic functions and to prove uniqueness of very weak solutions to a measure data problem
\begin{equation}
\label{eq:intro}
\begin{cases}-\dv\opA(x,\nabla u)=\mu\  \text{in} \ \Omega\\ u=0 \ \text{on} \ \partial\Omega,
\end{cases}
\end{equation}
where the leading part of the operator $\opA:\Omega\times\rn\to\rn$ is measurable with respect to the first variable and with respect to the second one exposes fully anisotropic growth expressed by the means of an $N$-function $\Phi:\R^n\to\rp0$.  Having an irregular datum $\mu$, one cannot expect existence of weak solutions, but the problem~\eqref{eq:intro} can be well-posed for various notions of very weak solutions --- renormalized, entropy, approximable solutions to name a few, cf.~\cite{IC-pocket}. Very weak solutions to problems with Orlicz growth attract substantial attention lately \cite{ACCZG,IC-measure-data,CGZG,CGZG-Wolff,DF,FiPr,gw-e,pgisazg1}. The sharp assumptions on $\mu$ to ensure uniqueness for these type of problems are still not known for any notion of solution, even when $\opA$ has standard $p$-growth with the classical instance of the $p$-Laplacian, see~\cite{BGO,DMMOP1}. The absolute continuity of the measure with respect to relevant capacity is treated as a natural condition ensuring uniqueness therein. Therefore, supplying the recent developments~\cite{ACCZG} carried out within the related anisotropic theory, we prove uniqueness for a broad measure data problems. In order to justify that the class is natural, we develop proper notions of generalized anisotropic Orlicz--Sobolev capacities. 

On the other hand, to justify the effort made to conduct analysis in this unconventional functional setting, we present a relevant example of a function generating {\em essentially fully anisotropic space} of Orlicz--Sobolev--type. As a matter of fact, we give a method of obtaining {\em essentially fully anisotropic Young function} that not only does not have othotropic decomposition, but also is not comparable to any function that can be transformed affinely into an orthotropic function, see Example~\ref{ex:ess-fully-aniso} in Section~\ref{sec:plane}. Additionally, it is (up to equivalence) trapped between $|t|^p$ and $|t|^{p}\log^\alpha(1+|t|)$ for any $p\geq 1$ and $\alpha>0$.

\subsection{Full anisotropy} Our analysis is settled within  fully anisotropic Orlicz spaces, where the norm is defined by the means of the functional\[\xi\mapsto\int_\Omega \Phi(\xi)\,dx,\]
where $\Phi:\R^n\to\rp0$, $n\geq 2$, is a \emph{fully anisotropic $n$-dimensional Young function,} that is even, convex function $\Phi:\rn\to [0,\infty]$, such that, $\Phi(0) = 0$ and  $\{\xi \in \rn: \Phi(\xi) \leq t\}$ is a compact set containing $0$ in its interior for every $t>0$. The function $\Phi$ is called an
\emph{$n$-dimensional $N$-function} if it is a fully anisotropic $n$-dimensional Young function and, in addition, $\Phi$ is finite-valued,
vanishes only at $0$, and  $\lim _{|\xi| \to
0}{\Phi (\xi)}/{|\xi|}=0\ $ and $\ \lim_{|\xi|\to\infty}\Phi(\xi)/|\xi|=\infty$.

An $n$-dimensional Young function $\Phi:\rn\to[0,\infty)$ is called {\em isotropic} if $\Phi(\xi)=\psi(|\xi|)$ with a classical Young function $\psi:[0,\infty)\to[0,\infty)$ and {\em anisotropic} if its dependence on $\xi$ is allowed to be more complicated. The easy instance of an anisotropic $n$-dimensional $N$-function is represented by power functions with different exponents \[\Phi\big((\xi_1,\dots,\xi_n)\big)=\sum_{i=1}^n |\xi_i|^{p_i},\] but we explain further that it can be  surprisingly much more robust. Anisotropic function is \emph{not necessarily} admitting the decomposition called {\em orthotropic} (studied e.g. in~\cite{BB,DFG}): \[\text{$\Phi\big((\xi_1,\dots,\xi_n)\big)=\sum_{i=1}^n\psi_i(|\xi_i|)\quad$ with Young functions $\ \psi_i:\rp0\to\rp0$.}\]  It does {\em not} even have to satisfy the monotonicity formula:
\begin{equation}
\label{monotonicity-property}
\text{if}\quad\xi=(\xi_1,\dots,\xi_n), \ \eta=(\eta_1,\dots,\eta_n),\ \text{ and }\ |\xi_i|\leq|\eta_i|,\ \ \text{ then }\ \ \Phi(\xi)\leq \Phi(\eta).
\end{equation}
Actually, it suffices to take $\Phi:\R^2\to\rp0$ given by
\[\Phi(\xi)=|\xi^1|^2+|\xi^2|^2+|\xi^1-\xi^2|^2\exp(|\xi^1-\xi^2|).\]
Indeed, \eqref{monotonicity-property} is violated since for $(2,0),(3,3)\in\R^2$ we have 
\[\Phi((2,0))=4(1+\exp(2))>4\cdot 5>18=\Phi((3,3)).\]
Moreover, an anisotropic $n$-dimensional $N$-function $\Phi$ can be not-comparable to any function satisfying the monotonicity  of the above form, in turn making also the generated by it functional space of Orlicz-Sobolev type {\em essentially anisotropic} and deprived from multiple handy properties. Let us comment on one more reason why this is non-trivial. The obvious example of a  function not admitting  the orthotropic decomposition is $\Phi(x,y)=(\max\{|x|,|y|\})^2$, but then $\Phi(x,y)$ is comparable to $|x|^2+|y|^2$. Since the Orlicz spaces generated by $ (\max\{|x|,|y|\})^2$ and $|x|^2+|y|^2$ are the same, we identify these functions (as in~\eqref{equiv}). There exists already known example by Trudinger~\cite{Tr}
\[\Phi(x,y)=|x-y|^\alpha+|x|^\beta\log^\delta(c+|x|),\qquad \alpha,\beta\geq 1,\] and $\delta\in \R$ if $\beta>1$, or $\delta>0$ if $\beta=1$, with $c> 1$ large enough to ensure convexity. Of course, such a function can be affinely transformed to the orthotropic one. 

We shall call a Young function {\em essentially fully anisotropic} if after any linear and invertible change of variables the orthotropic decomposition is impossible even up to equivalence (Definition~\ref{def:2} in Section~\ref{sec:plane}).  Despite anisotropic spaces  are considered throughout decades already in various contexts starting from~\cite{Ioffe,Klimov76,Sk1,Tr}, then~\cite{Cfully}, and from the point of view of partial differential equation they have received an attention lately, e.g.~\cite{An0,ACCZG,An,barletta,barlettacianchi,Ci-sym,pgisazg1,JLPS,S}, to our surprise, we were not able to find in the literature a~favorable example justifying development of the general framework and giving intuition how such functions behave.  We address this issue in Section~\ref{sec:plane} by giving an example of a Young function $\Phi$ living on the plane and being not comparable to any function having directional decomposition. Example~\ref{ex:ess-fully-aniso} provides a relevant function, whose construction is based on an inductional procedure involving three competing Young functions. In fact, the basic idea of construction of essentially fully anisotropic is to consider
\[\Phi(x,y)=\phi_1(x)+\phi_2(y)+\phi_3(x-y),\]
with a triple of competing $1$-dimensional Young functions $\phi_1, \phi_2, \phi_3:\R\to\rp0$, such that $\phi_i$ is not comparable with $\phi_j+\phi_k$ for any distinct $1 \leq i,j,k \leq 3$.  These functions compete and, consequently, while observing the images of balls with increasing radii,  the leading direction (the direction of the quickest growth) of the anisotropic $N$-function changes infinitely many times. Additionally, our essentially fully anisotropic $\Phi$ from Example~\ref{ex:ess-fully-aniso} satisfies $|\xi|^p \lesssim\Phi(\xi)\lesssim |\xi|^p\log^\alpha(1+|\xi|)$ with arbitrary $p\geq 1$, $\alpha>0$. The example and estimates of its sublevel sets is provided in Section~\ref{sec:plane}.

\subsection{Measure data problems. } We consider~\eqref{eq:intro} with a finite signed Radon  measure $\mu$, a~vector field $\opA:\Omega\times \rn\to\rn$ is a Caratheodory's function, which is monotone in the sense that for every
$\xi,\eta\in\rn$ such that $\xi\neq\eta$ we have
\begin{equation}\label{mon} (\opA(x,\xi) - \opA(x, \eta)) \cdot (\xi-\eta)> 0.
\end{equation}
We assume that $\opA$ satisfies growth and coercivity expressed by some fully anisotropic $n$-dimensional $N$-function $\Phi:\rn\to\rp0$ through the following conditions  
\begin{flalign}
\opA(x,\xi)\cdot \xi &\geq \Phi(c_1^\Phi\,\xi),\nonumber
\\
c_2^\Phi\widetilde{\Phi}(c_3^\Phi\, \opA(x, \xi))&
\leq  \Phi(c_4^\Phi\,\xi) +h(x),\label{constants}
\end{flalign}
for every  $\xi\in\rn$, a.e. $x\in\Omega$, with some constant $c_1^{\Phi},c_2^\Phi,c_3^\Phi,c_4^\Phi>0$ and a function $0\leq h\in L^1(\Omega)$. Here $\wt\Phi$ denotes the Young conjugate defined in~\eqref{wtPhi}. Existence of so-called {\em Approximable Solutions} and their anisotropic regularity in generalized Marcinkiewicz-type scale for~\eqref{eq:intro} is elaborated in~\cite{ACCZG}.  The results of~\cite{ACCZG} are provided separately in two cases -- for fast growing $\Phi$, namely \begin{equation}
\label{fast-growth-1}\int^\infty \frac{\wt\Phi_\circ(t)}{t^{1+n'}}\,dt=\infty\end{equation}  (for $\Phi_\circ$ being ``average in measure'' of $\Phi$ defined in Section~\ref{ssec:spaces}) the solutions are proven to exist in the {\em weak} sense, while for slowly growing they are {\em approximable}. This condition reflects the case of $W^{1,p}(\Omega)$ with $p>n$. As a matter of fact, making use of~\cite[Theorem~1a]{Ci-cont} and the embedding of~\cite{Cfully}, we note that \eqref{fast-growth-1} holds if and only if any function from Orlicz-Sobolev space $W^{1}L^{\Phi} (\Omega)$ has a~representative that is bounded and continuous. Therefore, in our analysis without loss of generality we shall consider the converse, that is\begin{equation}
\label{slow-growth-1}
\int^\infty \frac{\wt\Phi_\circ(t)}{t^{1+n'}}\,dt<\infty.
\end{equation} 
 If $\mu=f\in L^1(\Omega)$, a function $u \in W_0^{1}\mathcal{L}^{\Phi} (\Omega)$ is called a weak
solution to \eqref{eq:intro}  if
\begin{equation*}
 \int_\Omega \opA(x, \nabla u) \cdot \nabla
\varphi\, dx = \int_\Omega f\varphi\,dx\quad\text{
for every }\   \varphi \in W_0^{1}\mathcal{L}^{\Phi} (\Omega) \cap
L^{\infty}(\Omega).
\end{equation*}

Of course if the right-hand side datum is not regular enough one cannot expect weak solutions to exist. Before we pass to the very weak solutions, let us point out an obstacle. Distributional solutions to equation $-\Delta_p u=\mu$ when $p$ is small ($1<p<2-1/n$) do not necessarily belong to $W^{1,1}_{loc}(\Omega)$. The easiest example to give is the fundamental solution (when $\mu=\delta_0$). This restriction on the growth can be dispensed by the use of a~weaker derivative. Having the symmetric truncation at the level $k$, denoted as $T_k$, is defined in~\eqref{trunc}. Let us denote by $\mathcal{T}^{1}_0L^\Phi(\Omega)$ the space of measurable functions, such that for every $k>0$ it holds that $T_k(u)\in {W}^{1}_0L^{\Phi}(\Omega)$. For every $u\in \mathcal{T}^{1}_0L^{\Phi}(\Omega)$, we assign a generalized gradient obtained as a limit of gradients of truncations of $u$, namely $\nabla u:=\lim_{k\to\infty}\nabla(T_k(u))$.

By {\em Approximable Solution} to problem \eqref{eq:intro} under the above described regime we mean a~function $u\in \mathcal{T}^{1}_0L^{\Phi}(\Omega)$, if there exists a~sequence $\{h_s \}\subset C_0^\infty (\Omega)$ such that \[\text{$h_s\to \mu$  weakly-$\ast$ in the space of
measures}\] cf.~\eqref{meas1}, such that
the sequence of~weak solutions $\{u_s\}\subset W_0^{1}\mathcal{L}^{\Phi} (\Omega)$ to problems
\begin{equation}\label{prob:trunc} \begin{cases}
-\dv \, \opA(x,\nabla u_s) = h_s &\quad \mathrm{ in}\quad  \Omega\\
u_s =0 &\quad \mathrm{  on}\quad \partial\Omega,
\end{cases}
\end{equation}
satisfies
\begin{equation}
 \label{ae}
 u_s\to u\qquad  \text{a.e. in }\Omega.
 \end{equation}
If the limit function $u$ does not depend on the choice of the approximate sequence $\{h_s\}$, then we say that $u$ is unique.
We recall the formulation of the existence result in Proposition~\ref{prop:ex-sola}. In~\cite{ACCZG} the uniqueness is proven only when $\mu$ is absolutely continuous with respect to Lebesgue's measure. We extend it to the class of measures admitting a decomposition. Below we explain that, when the involved Orlicz-Sobolev space is reflexive, this result entails uniqueness precisely for the class of measures that do not charge sets of anisotropic capacity zero.

\begin{theorem}[Uniqueness I]
\label{theo:uniq} Suppose that $\Omega$ is a bounded Lipschitz domain in $\rn$, $\opA:\Omega\times\rn\to\rn$ satisfies~\eqref{mon} and~\eqref{constants}  with a~fully anisotropic $n$-dimensional $N$-function $\Phi,$ which grows slow enough to satisfy~\eqref{slow-growth-1}, and $\mu\in\Mb(\Omega)$ is such that $\mu=f+\dv G$ in the sense of distributions with some $f\in L^1(\Omega)$ and $G\in L^{\wt \Phi}(\Omega;\rn)$.  Assume futher that  $\{u_s\}$  is a sequence of weak solutions to problems~\eqref{prob:trunc} with data $\{h_s\}\subset C_0^\infty(\Omega)$ with $h_s=f_s+\dv G_s$, where $\{f_s\}$, $\{ {G_s}\}$ are sequences of bounded functions such that \[\text{$f_s\to f$ strongly in $L^1(\Omega)\quad$ and $\quad G_s\to G$ modularly in $L^{\wt \Phi}(\Omega;\rn)$.}\] Then the approximable solution $u$ to~\eqref{eq:intro} obtained as a.e. limit of  $\{u_s\}$  is unique.
\end{theorem}

\subsection{Measure characterization } To characterize fully the relevant diffuse measures we need to develop the capacity framework. We do it in the natural case -- in the whole class of reflexive anisotropic Orlicz-Sobolev spaces. For related results on fine behaviour of Orlicz-Sobolev functions in isotropic spaces see~\cite{Ci-cont,MSZ}.

\medskip

 A bounded total variation measure can be decomposed into two parts: absolutely continuous and singular with respect to fully anisotropic Orlicz--Sobolev capacity (Lemma~\ref{lem:basic-decomp}). In order to present the decomposition of measures not charging anisotropic capacities, let us denote  by $\Mb(\Omega)$ the set of measures  of bounded total variation in  $\Omega\subset \R^n$. Anisotropic capacity $\capP$ is defined in Section~\ref{ssec:sob-cap}. By $\MP(\Omega)$ we mean a set of $\Phi$-diffuse measures (called also $\Phi$-soft measures), that is such $\mup\in\Mb$ that for any set in $\rn$ of zero $\mup$-measure its capacity $\capP$ is also zero. One can think that a measure $\mup\in \MP(\Omega)$ is a bounded measure diffuse `absolutely continuous' with respect to $\capP$.  We have the following theorem.

\begin{theorem}[Characterization of measures]\label{theo:decomp} Suppose a measure $\mup\in\Mb(\Omega)$ is defined on a bounded set $\Omega\subset\rn,$ $n\geq 2$ and $\Phi\in\Delta_2\cap\nabla_2$ is a fully anisotropic $n$-dimensional $N$-function. Then \[\mu_\Phi\in\MP(\Omega)\quad\text{if and only if}\quad \mu_\Phi\in L^1(\Omega)+(W_0^{1,\Phi}(\Omega))',\] i.e. there exist
$f\in L^1(\Omega)$ and $G \in L^{\wt{\Phi}}(\Omega;\rn)$, such that $\mu_\Phi=f+\dv G$ in the sense of~distributions. 
\end{theorem}

\begin{remark}If $\Phi$ grows so fast that \eqref{slow-growth-1} is violated (i.e. when it satisfies condition generalizing $p>n$ for $\Phi_p(\xi)=(\sum_{i=1}^n|\xi_i|^2)^\frac{p}{2}$ generating $W^{1,p}$), then points have positive capacity. Consequently, all bounded measures are $\Phi$-diffuse.
\end{remark}

\begin{remark}The decomposition of Theorem~\ref{theo:decomp} cannot be unique as $L^1(\Omega)\cap(W_0^{1,\Phi}(\Omega))'\neq\{0\}.$ 
\end{remark}

\begin{corollary}\label{coro:decomp}
 For every $\mu\in\Mb(\Omega)$ there exists a decomposition 
\[\mu=\mu_\Phi +\mu_s^+-\mu_s^-\]
with some $\mu_\Phi$ which does not charge sets of anisotropic capacity zero, disjoint sets $E_-$ and $E_+$ of anisotropic capacity zero and nonnegative measures $\mu_s^-,\mu_s^+\in\Mb(\Omega)$ concentrated on $E_-$ and $E_+$, respectively, see Lemma~\ref{lem:basic-decomp}. Thus, any $\mu\in\Mb(\Omega)$ admitts a~decomposition 
\[\mu=f+\dv G +\mu_s^+-\mu_s^-\]
with $f,G$ as in Theorem~\ref{theo:decomp}.
\end{corollary}

Let us comment the result starting with presenting isotropic consequences. The Orlicz part was not known before, the power growth part is classical and retrieved in detail. 
\begin{corollary} Suppose $A:[0,\infty)\to[0,\infty)$ is a doubling Young function. Then  $\mu_A\in \Mb(\Omega)$ does not charge the sets of Sobolev $A$-capacity zero if and only if $\mu_A\in L^1(\Omega)+(W_0^{1,A}(\Omega))'$, i.e. there exist $f\in L^1(\Omega)$ and $G \in (L^{\wt{A}}(\Omega))^n$, such that $\mu_A=f+\dv G$ in the sense of distributions. In particular, the special case of this result is the classical measure characterization~\cite{BGO,DaM,DMMOP1}:   $\mu_p\in \Mb(\Omega)$ does not charge the sets of the classical Sobolev $p$-capacity zero if and only if $\mu_p\in L^1(\Omega)+W^{-1,p'}(\Omega)$, i.e. there exist $f\in L^1(\Omega)$ and $G \in (L^{p'}(\Omega))^n$, such that $\mu_p=f+\dv G$ in the sense of distributions.
\end{corollary} 

 To prove Theorem~\ref{theo:decomp}, we develop the framework of capacities in this unconventional space setting. In particular, Proposition~\ref{prop:quasicont} yields that any fully anisotropic Orlicz-Sobolev function has a representative, which is quasicontinuous with respect to properly defined anisotropic capacity. This part is related to earlier studies in the isotropic setting~\cite{MSZ,bie}, in metric measure spaces~\cite{BO,Tuo}, as well as within generalized Orlicz framework~\cite{bahaha}.

\subsection{Conclusion for measure data problems}\label{ssec:con-meas-data-probs}

Proposition~\ref{prop:ex-sola} together with Theorems~\ref{theo:uniq} and~\ref{theo:decomp} have the following direct consequence in the reflexive case for problems with $\Phi$-diffuse measures.  Note that in the reflexive case approximable solutions coincide with the distributional ones and, if the datum allows, with weak ones.
\begin{theorem}[Uniqueness II]
\label{theo:uniq-refl}
If $\Omega$ is a bounded Lipschitz domain in $\rn$, $\opA$ satisfies~\eqref{mon} and \eqref{constants}  with a fully anisotropic $n$-dimensional $N$-function $\Phi$, which satisfies $\Delta_2\cap\nabla_2$-conditions near infinity, $\mu\in \MP(\Omega)$, then there exists a unique approximable solution to~\eqref{eq:intro}.
\end{theorem}
 The above result is particularly meaningful if $\Phi$ is growing so slow that~\eqref{slow-growth-1} holds. Otherwise, however, is not excluded in the statement as all measures are $\Phi$-diffuse and belong to $(W_0^{1,\Phi}(\Omega))'$. Consequently, then approximable solutions are weak and unique.

\subsection{Organization of the paper. } Section~\ref{sec:plane} is devoted to the method of obtaining examples of essentially fully anisotropic Young function on the plane (Example~\ref{ex:ess-fully-aniso}). Section~\ref{sec:prelim} introduces to functional framework necessary to carry analysis of measure data problems. In Section~\ref{sec:uniq} we prove uniqueness (Theorem~\ref{theo:uniq}), while basic analysis on anisotropic capacities together with the proof of measure decomposition (Theorem~\ref{theo:decomp}) are provided  in Section~\ref{sec:cap}. Theorem~\ref{theo:uniq-refl} follows as a direct corollary of these results.
 
\section{Example of essentially fully anisotropic Young function on the plane}\label{sec:plane}
To define essential full anisotropy we need to introduce the equivalence classes of functions.  Let $F,G:\rn \to \R$. We say that $F$ \emph{dominates} $G$ ($F \gtrsim G$ for short) if
there exist constants $c_1,d_1>0$ such that
\[
	c_1\, G(d_1 x) \leq F(x), \qquad x \in \rn.
\]We say that $F,G:\rn \to \R$ are \emph{equivalent} ($F \simeq G$ for short) if they dominate each other, that is if
there exist constants $c_1,c_2,d_1,d_2>0$ such that
\begin{equation}\label{equiv}
	c_1 F(d_1 x) \leq G(x) \leq c_2 F(d_2 x), \qquad x \in \rn.
\end{equation}
We also say that $F,G$ are \emph{incomparable} 
\begin{equation}\label{incomp}
\text{if neither $\ F \gtrsim G\ $ nor $\ G \gtrsim F\ $ holds.}
\end{equation}

\begin{definition}\label{def:2}
We say that a Young function $\Phi:\rn \to [0,\infty]$ is \emph{essentially fully anisotropic} if there exists no linear invertible map $T:\R^n \to \R^n$ such that \[\Phi(T(x_1,\dots,x_n)) \simeq \sum_{i=1}^n \psi_i(|x_i|)\] for some Young functions $\psi_i:\rp0 \to [0,\infty]$, $i=1,\dots,n$.
\end{definition}

\noindent Let us start with some auxiliary facts and lemmas (the proofs are straightforward and thus we shall omit them).

\begin{fact}\label{fact:1}
Let $F,G,H:\R^n \to \R$. Then
\begin{itemize}
\item[(a)] If $F \simeq G$, then $G \simeq F$.
\item[(b)] If $F \simeq G$ and $G \simeq H$, then $F \simeq H$.
\end{itemize}
\end{fact}

\begin{fact}\label{fact:2}
Let $F_1,F_2,G_1,G_2:\rn \to \R$ are radially increasing. Then $F_1 \simeq G_1$ and $F_2 \simeq G_2$ implies $F_1+F_2 \simeq G_1 + G_2$.
\end{fact}

\noindent From now on we focus on functions living on the plane. We have the following observation.

\begin{lemma}\label{lem:1}
Suppose $\Phi:\R^2 \to [0,\infty)$ and $f,g:\R \to [0,\infty)$ are even radially increasing functions such that $\Phi(x,y) \simeq f(x)+g(y)$, $f(0)=0$, and $g(0)=0$. Then $f(x) \simeq \Phi(x,0)$ and $g(y) \simeq \Phi(0,y)$. Moreover, $\Phi(x,y) \simeq \Phi(x,0)+\Phi(0,y)$.
\end{lemma}
\begin{proof}
By taking $x=y=0$ we get $g(0)=0$. Take $y=0$ to get $f(x) \simeq \Phi(x,0)$ and $x=0$ to get $g(y) \simeq \Phi(0,y)$. Thus from Fact \ref{fact:1} and  \ref{fact:2} we get $\Phi(x,y) \simeq f(x)+g(y) \simeq \Phi(x,0)+\Phi(0,y)$. 
\end{proof}

\noindent The following proposition ensures that the function we construct is indeed fully anisotropic. 
 
\begin{proposition}\label{thm:1}
Suppose $\phi_1,\phi_2,\phi_3:\R \to [0,\infty)$ are $1$-dimensional Young functions such that $\phi_i$ is incomparable with $\phi_j+\phi_k$ for any distinct $1 \leq i,j,k \leq 3$. Then a Young function \begin{equation}
\label{eq:Phi}
\Phi(x,y)=\phi_1(x)+\phi_2(y) + \phi_3(x-y)
\end{equation} is essentially fully anisotropic  in the sense of Definition~\ref{def:2}.
\end{proposition}

\begin{proof}
Assume by contradiction that there exist real $a,b,c,d$ such that the matrix
\[
	T = \left[ \begin{array}{cc}
	a & b \\ c & d
\end{array}	 \right]
\]
is invertible and such that $\Phi(T(x,y)) \simeq f(x)+g(y)$ for some even convex $f,g:\R \to [0,\infty)$. From Lemma \ref{lem:1} we get that $\Phi(T(x,y)) \simeq \Phi(x,0)+\Phi(0,y)$. Since \[T(x,y)=(ax+by,cx+dy),\] this gives
\begin{align*}
	 & \phi_1(ax+by)  + \phi_2(cx+dy) + \phi_3((a-c)x + (b-d)y) \\
	 & \qquad \qquad \simeq  \qquad  \phi_1(ax)  + \phi_2(cx) + \phi_3((a-c)x) + \phi_1(by)  + \phi_2(dy) + \phi_3((b-d)y).
\end{align*}
Assume $a \ne 0$ and $b \ne 0$. Then restricting to the line $y=-(a/b)x$ we get
\begin{align*}
	 &  \phi_2\left( \left( c - \tfrac{da}{b} \right) x \right) + \phi_3\left(\left( (a-c) - (b-d) \tfrac{a}{b} \right) x \right) \\
	 & \qquad \qquad \simeq  \qquad  \phi_1(ax)  + \phi_2(cx) + \phi_3((a-c)x) + \phi_1(ax)  + \phi_2\left( \tfrac{da}{b} \right) + \phi_3\left((b-d) \tfrac{a}{b} x \right).
\end{align*}
Note that $c-\frac{da}{b} \ne 0$ and $(a-c)-(b-d)\frac{a}{b} \ne 0$ since $\det(T) \ne 0$.
In particular $\phi_2 + \phi_3 \gtrsim \phi_1$, which is the desired contradiction. Playing the same game with $c$ and $d$ we get that $\phi_1 + \phi_3 \gtrsim \phi_2$ under $c \ne 0$ and $d \ne 0$. 

Assume now that $ab=0$ or $cd=0$. By exchanging the roles of $x$ and $y$ we can assume that $b=0$ and $a \ne 0$ (if $a=b=0$ then $T$ is not invertible). We can also assume that $c=0$ and $d \ne 0$ (if $d=0$ then $T$ is again not invertible). Using $b=c=0$ we get  
\[
	  \phi_1(ax)  + \phi_2(dy) + \phi_3(ax  -dy) \quad \simeq  \quad  \phi_1(ax)  + \phi_3(ax)   + \phi_2(dy) + \phi_3(dy).
\]
Now, taking $y=(a/d)x$ yields $\phi_1 + \phi_2 \gtrsim \phi_3$, which again is a contradiction.
\end{proof}

\noindent We remark that in order to be essentially fully anisotropic it is not enough for a function to have a form~\eqref{eq:Phi}. 

\begin{remark}\label{rem:powers} Function $\Phi:\R^2\to\rp0$ given by $\Phi(x,y)=|x|^p+|x-y|^q+|y|^r$ for abitrary $p,q,r>0$ is not fully anisotropic in the sense of Definition~\ref{def:2}.
 
Indeed, changing variables if necessary we can assume that $p \leq q \leq r$. Clearly, 
\[\Phi(x,y) \lesssim |x|^p + |y|^r + |x|^q + |y|^q.\]
We shall show the opposite inequality. \\

\noindent {\it Case i)}  $\tfrac 13\leq \tfrac {|x|}{|y|}\leq 3$. 
 In this case $|x|^q\lesssim |x|^p + |x|^r$ (since $p \leq q \leq r$). Consequently, \begin{flalign*}
\Phi(x, y) &= |x|^p + |x - y|^q + |y|^r \geq |x|^p + |y|^r\gtrsim 
|x|^p + |x|^r\\
& \gtrsim 
|x|^p + |x|^r + |x|^q \gtrsim 
|x|^p + |y|^r + |x|^q + |y|^q.\end{flalign*}

\noindent {\it Case ii)} $|x| \geq 3|y|$ or $|y| \geq 3|x|$. In both cases $|x - y| \geq \tfrac 12 |x| + \tfrac 12 |y|$. Let us check the first case
(the other one we get by exchanging $x$ and $y$). We have
\[|x - y| \geq |x| - |y| = \tfrac 12|x| +\tfrac 12 |x| - |y| \geq \tfrac 12|x| + \tfrac 32|y| - |y| =\tfrac 12( |x| + |y|).\]
Thus, in this case
\begin{flalign*}\Phi(x, y) &= |x|^p + |x - y|^q + |y|^r\gtrsim
|x|^p + (|x| + |y|)^q + |y|^r\\
&\gtrsim
|x|^p + |y|^r + |x|^q + |y|^q.\end{flalign*}
\end{remark}

We are in position to construct an essentially fully anisotropic Young function in fact controlling also its growth.
\begin{example}\label{ex:ess-fully-aniso}
Let $p \geq 1, \alpha > 0$ and let $\phi_-(t)=|t|^p$ and $\phi_+(t)=|t|^p \log^\alpha(|t|+1)$. There exist  $\phi_1,\phi_2,\phi_3:\R \to [0,\infty)$ being $1$-dimensional Young functions such that $\phi_i$ is incomparable with $\phi_j+\phi_k$ for any distinct $1 \leq i,j,k \leq 3$ and 
\[ \phi_- = \min(\phi_1, \phi_2, \phi_3) \leq \phi_1, \phi_2, \phi_3\leq \max(\phi_1, \phi_2, \phi_3)=\phi_+.
\]
 Then $\Phi:\R^2\to\rp0$ given by $\Phi(x,y)=\phi_1(x)+\phi_2(y) + \phi_3(x-y)$  is an essentially fully anisotropic Young function  in the sense of Definition~\ref{def:2}.  

In order to construct it, we observe first that obviously the function $\phi_+$ is convex. 
We shall construct $\phi_i$ inductively. 
We set $t_0=1$ and $\phi_1=\phi_2=\phi_-$ and $\phi_3=\phi_+$ on $[0,t_0]$. Suppose our functions are already defined on $[0,t_k]$ with some $t_k \geq k$ and that at $t=t_k$ one of these functions is equal to $\phi_+$ and the other two coincide with $\phi_-$. We shall define $\phi_1, \phi_2, \phi_3$ on $[t_k, t_{k+1}]$ where $t_{k+1} \geq k+1$. Let us assume that $\phi_3(t_k)=\phi_+(t_k)$ (if $\phi_2(t_k)=\phi_+(t_k)$ in what follows we substitute $(\phi_1, \phi_2, \phi_3) \to (\phi_3, \phi_1, \phi_2)$ and if $\phi_1(t_k)=\phi_+(t_k)$ then we substitute $(\phi_1, \phi_2, \phi_3) \to (\phi_2, \phi_3, \phi_1)$). Let $l_k$ be the linear function whose graph passes through $(t_k, \phi_-(t_k))$ and is tangent to the graph of $\phi_+$ (say, in point $(h_k,\phi_+(h_k))$ with $h_k>t_k$). We set $\phi_2=l_k$ and $\phi_3=\phi_+$ on $[t_k,h_k]$. Let $s_k>h_k$ be the point satisfying $l_k(s_k)=\phi_-(s_k)$. We set $\phi_3=l_k$ and $\phi_2=\phi_+$ on $[h_k,s_k]$. We also take $\phi_1=\phi_-$ on $[t_k,s_k]$. We have defined $\phi_1, \phi_2, \phi_3$ on $[t_k,s_k]$. Let $t_{k+1} > \max(k+1, s_k)$ be a point such that $\log^\alpha(|t_{k+1}|+1) \geq k^{p+1}$. This ensures that $\phi_+(t_{k+1}) \geq 2k \phi_-(t_{k+1} k)$. If we now take $\phi_1=\phi_3=\phi_-$ and $\phi_2=\phi_+$ on $[s_k, t_{k+1}]$, then $\phi_2(t_{k+1}) \geq k \phi_1(t_{k+1} k)+k \phi_3(t_{k+1} k)$.  This condition (considered for every $k$) ensures that $\phi_1+\phi_3$ does not dominate $\phi_2$. As a~result the triple $\phi_1, \phi_2, \phi_3$ is good. From the construction we have $\max(\phi_1, \phi_2, \phi_3)=\phi_+$ and $\min(\phi_1, \phi_2, \phi_3)=\phi_-$. We reach the conclusion by the use of Proposition~\ref{thm:1}.
\end{example}

The critical role for embeddings of anisotropic Sobolev-Orlicz spaces into Orlicz spaces is played by the isotropic function that has he same  sublevel sets as the one governing gradient. Let us estimate those of our essentially fully anisotropic example function.

\begin{lemma}\label{lem:level-sets}
Let $p \geq 1$, $\alpha>0$. Suppose $|x|^p \leq \phi_1(x), \phi_2(x), \phi_3(x) \leq \phi_+(x)$ are non-negative even convex $C^1$ functions vanishing at the origin and such that $\max\{\phi_1, \phi_2, \phi_3\}= \phi_+$. Then 
\begin{align*}
	\tfrac{\pi}{4}(\phi_+^{-1}&(t/3))^2  \leq \tfrac14 |\{u \in \R^2: \  \phi_+(|u|) \leq \tfrac13 t \}| \\ &  \leq |\{(x,y): \  \phi_1(x)+\phi_2(y)+\phi_3(x-y) \leq t \}| \leq \tfrac{4p}{p+1} \phi'_+(\phi_+^{-1}(t))^{\frac1p}  (\phi_+^{-1}(t))^{\frac1p+1}. 
\end{align*}
\end{lemma}

\begin{proof}
Suppose $\phi_1(t)=\phi_+(t)$. Then $\phi_1(x) \geq \phi_+(t)+(|x|-t)\phi'_+(t)$. We get
\begin{flalign*}
	|\{(x,y): & \  \phi_1(x)+\phi_2(y)+\phi_3(x-y) \leq \phi_+(t) \}| \\
	& \leq |\{(x,y): \  \phi_+(t)+(|x|-t)\phi'_+(t)+|y|^p \leq \phi_+(t) \}| \\
	& = |\{(x,y): \  |y|^p \leq (t-|x|)\phi'_+(t)  \}| = 4 \phi'_+(t)^{\frac1p} \int_0^t (t-x)^{\frac1p} \dd x = 4 \phi'_+(t)^{\frac1p} \tfrac{p}{p+1} t^{\frac1p+1}.
\end{flalign*}
The same holds true if $\phi_2(t)=\phi_+(t)$. Now, suppose $\phi_3(t)=\phi_+(t)$. Applying a determinant $1$ change of variables $u=x$, $v=x-y$ we get 
\[
	|\{(x,y):  \  \phi_1(x)+\phi_2(y)+\phi_3(x-y) \leq \phi_+(t) \}|  = |\{(u,v):  \  \phi_1(u)+\phi_2(u-v)+\phi_3(v) \leq \phi_+(t) \}| 
\]
and thus we arrive at the same estimate if we repeat the above argument. We have proved the upper bound.

The inequality
\[
	|\{(x,y): \  \phi_1(x)+\phi_2(y)+\phi_3(x-y) \leq t \}| \geq  |\{(x,y): \  \phi_+(x)+\phi_+(y)+\phi_+(x-y) \leq t \}|.
\]
is obvious. Now, by convexity of $\phi_+$ we have
\begin{flalign*}
	|\{(x,y): \ & \phi_+(x)+\phi_+(y)+\phi_+(x-y) \leq t \}| \\
	&\geq |\{(x,y): \  \phi_+(x)+\phi_+(y)+\phi_+(2x)+\phi_+(2y) \leq t \}| \\
	& \geq  |\{(x,y): \  \tfrac12\phi_+(2x)+ \tfrac12 \phi_+(2y)+\phi_+(2x)+\phi_+(2y) \leq t \}| \\
	& = |\{(x,y): \  \phi_+(2x)+ \phi_+(2y) \leq \tfrac23 t \}| 
	 = \tfrac14 |\{(x,y): \  \phi_+(x)+ \phi_+(y) \leq \tfrac23 t \}| \\ 
& \geq  \tfrac14 |\{(x,y): \  \phi_+((x^2+y^2)^{\frac12}) \leq \tfrac13 t \}| = \frac14 |\{u \in \R^2: \  \phi_+(|u|) \leq \tfrac13 t \}|\\
& = \tfrac{\pi}{4}\Big(\phi_+^{-1}\big(\tfrac{t}{3}\big)\Big)^2. 
\end{flalign*}\end{proof}

\begin{corollary}\label{coro:between} For the  even non-negative convex fully anisotropic function $f:\R^2 \to \R$ such that $|x|^p \leq f(x) \leq |x|^p \log^\alpha(1+|x|)$ constructed in in Example~\ref{ex:ess-fully-aniso} and for $t>1$ it holds that 
\[\frac{1}{C} t^{\frac2p} \log^{-\frac{2\alpha}{p}}(t+1) \ \leq \ |\{x \in \R^2: \ f(x) \leq t \}| \ \leq C \ t^{\frac2p} \log^{-\frac{\alpha}{p}}(t+1)\]
with some $C=C(\alpha,p)$.
\end{corollary}

\begin{proof}
Let $\phi_1, \phi_2, \phi_3$ be the functions constructed in Example~\ref{ex:ess-fully-aniso}. Define $f(x)=\phi_1(x)+\phi_2(y)+\phi_3(x-y)$ and $\phi_+(t)=|t|^p\log^\alpha(1+|t|)$. We shall now apply Lemma \ref{lem:level-sets}. To do this we observe that for $t>1$ we have
\begin{align*}
	\phi_+'(t)^{\frac1p} t^{\frac1p + 1} & = \left(  p t^{p-1}  \log^\alpha(t+1) + \alpha \frac{t^p}{t+1} \log^{\alpha-1}(t+1) \right)^{\frac1p}  t^{\frac1p + 1}  \\
	& \leq \left(  p^{\frac1p} t^{1-\frac1p}  \log^{\frac{\alpha}{p}}(t+1) + \alpha^{\frac1p} \frac{t}{(t+1)^{\frac1p}} \log^{\frac{\alpha-1}{p}}(t+1) \right)  t^{\frac1p + 1} \\
	& \leq   p^{\frac1p} t^{2}  \log^{\frac{\alpha}{p}}(t+1) + \alpha^{\frac1p} t^2 \log^{\frac{\alpha-1}{p}}(t+1) \leq c(\alpha,p) t^2 \log^{\frac{\alpha}{p}}(t+1).
\end{align*}
We also claim that for $t>1$
\[
	c_1 t^{\frac1p} \log^{-\frac{\alpha}{p}}(t+1) \leq \phi_+^{-1}(t) \leq c_2 t^{\frac1p} \log^{-\frac{\alpha}{p}}(t+1)
\]
with some $c_i=c_i(\alpha,p)$, $i=1,2$. To this end, due to the monotonicity of $\phi_+$, it suffices to show that
\[
	\phi_+(c_2 t^{\frac1p} \log^{-\frac{\alpha}{p}}(t+1)) \geq t \geq \phi_+(c_1 t^{\frac1p} \log^{-\frac{\alpha}{p}}(t+1)).
\] 
Now it suffices to observe that
\[
	c_2^p  \log^{-\alpha}(t+1) \log^\alpha(1+c_2 t^{\frac1p} \log^{-\frac{\alpha}{p}}(t+1)) \xrightarrow[t \to \infty]{} c_2^p p^{-\alpha} .
\]
It follows that 
\[
	\frac{\pi}{4}(\phi_+^{-1}(t/3))^2 \geq \frac{\pi}{4} c_1^2 t^{\frac2p} \log^{-\frac{2\alpha}{p}}(t+1)
\]
and
\begin{flalign*}
	 \phi'_+(\phi_+^{-1}(t))^{\frac1p}  (\phi_+^{-1}(t))^{\frac1p+1} &\leq c(\alpha,p) (c_2 t^{\frac1p} \log^{-\frac{\alpha}{p}}(t+1))^2 \log^{\frac{\alpha}{p}}(1+c_2 t^{\frac{1}{p}} \log^{-\frac{\alpha}{p}}(t+1))\\
	 & \leq C(\alpha,p) t^{\frac{2}{p}} \log^{-\frac{\alpha}{p}}(t+1).
\end{flalign*}
\end{proof}

\begin{remark}
W note that the above function $f$ is not coordinate-wise monotone on $\R_+^2$. Indeed take $x>0$ and consider points $(0,x)$ and $(x,x)$. Then  coordinate-wise monotonicity would imply the inequality $f(0,x) \leq f(x,x)$, which reduces to $\phi_3(x) \leq \phi_1(x)$, which is clearly not true for certain real numbers $x$ (there are points $x$ such that $\phi_1(x)=\phi_-(x)< \phi_+(x)=\phi_3(x)$).  In particular, it does not support the monotonicity property~\eqref{monotonicity-property}.
\end{remark}

\section{Preliminaries of the Measure Data part} \label{sec:prelim}
\subsection{Notation and fundamental definitions}\label{ssec:notation}In the sequel $\Omega$ is a bounded open set in $\R^n$,  $n\geq 2$. 
We shall make use of symmetric truncations of a real-valued function\begin{equation}
\label{trunc} T_k f=\max\{-k,\min\{f,k\}\}.
\end{equation}
By $\mu_1\ll\mu_2$ we denote we mean that $\mu_1$ is absolutely continuous with respect to $\mu_2$. Assume further that $\Phi$ is an $n$-dimensional  Young function as defined in Introduction.  $C _0(\Omega)$ denotes the
space of continuous functions with compact support in $\Omega$.  A sequence of functions $\{f_k\} \subset L^1(\Omega)$ is said to weak-$\ast$ converge to $\mu$ in the space of measures if
\begin{equation}\label{meas1}
\lim _{k \to \infty} \int _\Omega \varphi f_k \, dx = \int _\Omega \varphi
\,d\mu
\end{equation}
for every function $\varphi \in C _0(\Omega)$.

\medskip

\emph{The Young conjugate} of an $n$-dimensional Young function $\Phi$ is a function $\wt \Phi:\rn\to[0,\infty)$
defined as
\begin{equation}
\label{wtPhi}
\wt \Phi (\xi) = \sup\{\eta \cdot \xi - \Phi(\eta):\ \eta \in \R^n\}\quad \hbox{for $\xi \in \R^n$\,.}
\end{equation}
Note that if $\Phi$ is an $n$-dimensional Young function or $N$-function, its conjugate $\wt\Phi$ is of the same class. Moreover, Young conjugation is involute, i.e. $\wt{(\wt\Phi)}=\Phi$.

A typical condition imposed on Young functions to infer strong properties of the functional setting generated with their use is a doubling condition.  A  Young function $\Phi$ is said to satisfy the \emph{$\Delta _2$-condition near infinity}, briefly $\Phi \in \Delta _2$ near infinity,  if it is finite--valued and there
exist positive constants $c$ and $M$ such that $\Phi (2\xi) \leq c \Phi(\xi)$   if $|\xi|\geq M$. We say that $\Phi\in\nabla_2$ (near infinity) if $\wt\Phi\in\Delta_2$ near infinity. If both $\Phi,\wt\Phi\in\Delta_2$ near infinity we call $\Phi$ doubling near infinity. In such a case there exists $1<i_\Phi\leq s_\Phi<\infty$ such that for every fixed vector $\xi \in\partial B(0,1)$ and $0\ll t$
\[
\frac{\Phi(t\xi)}{t^{i_G}}\quad\text{is non-decreasing}\qquad\text{and}\qquad\frac{\Phi(t\xi)}{t^{s_G}}\quad\text{is non-increasing,}\]
see~\cite{barlettacianchi}.

\subsection{Fully anisotropic functional setting}\label{ssec:spaces} 
Classical contribution on anisotropic version of Orlicz-Sobolev spaces is \cite{Klimov76}, see also~\cite{Cfully,Sch,Sk1}.

\medskip

\textit{The  anisotropic Orlicz space. } Let  $\Phi$ be an $n$-dimensional Young function.  The  \textit{anisotropic Orlicz space}
$L^\Phi(\Omega;\rn)$ is the set of all measurable vector-valued functions $U$ such that
the norm
$$\|U\|_{L^\Phi (\Omega;\rn)} = \inf \bigg\{\lambda >0: \int _\Omega\Phi\big(\tfrac 1\lambda U \big)\, dx \leq 1\bigg\}$$
is finite. The space  $L^\Phi (\Omega;\rn)$, equipped with this norm,
is a Banach space. We distiguish two subclasses of $L^\Phi(\Omega;\rn)$ such that
\begin{equation} \label{EPP}
E^\Phi(\Omega;\rn) \subset  \mathcal L^\Phi(\Omega;\rn) \subset L^\Phi(\Omega;\rn),
\end{equation}
where the Orlicz class $\mathcal L^\Phi(\Omega, \rn)$ consists of such functions $U$ that $ \int _\Omega\Phi\big( U \big)\, dx<\infty$, whereas the space $E^\Phi(\Omega;\rn)$ is the
closure in $L^\Phi(\Omega;\rn)$ of the space of bounded functions. Clearly, both inclusions in~\eqref{EPP} hold as equalities if and only if $ \Phi\in\Delta_2$ near infinity.  In general, $E^\Phi(\Omega;\rn)$ is separable, but $L^\Phi(\Omega;\rn)$ does not have to be.

For every $U \in L^\Phi (\Omega;\R^n)$ and $V \in
L^{\widetilde\Phi} (\Omega;\R^n)$, we have the following H\"older-type inequality
\begin{equation}\label{holder}
\int _\Omega |U \cdot   V|\,dx \leq 2\|U\|_{L^\Phi (\Omega;\R^n)}
\|V\|_{L^{\widetilde\Phi} (\Omega;\R^n)}.
\end{equation}
In fact, $L^\Phi$ and $L^{\wt{\Phi}}$ are associate spaces, but they do not have to be dual to one another. In general,  if $\Phi$ is an arbitrary $n$-dimensional  $N$-function, then by\cite[Proposition 2.3]{AdBF1} we get that
\begin{equation}\label{Dual}
\hbox{the dual
of $E^{\Phi}(\Omega;\rn)$ is isomorphic and homeomorphic to
$L^{\widetilde{\Phi}}(\Omega;\rn)$.}
\end{equation}
The space $L^\Phi (\Omega; \R^n)$ is reflexive and separable, if and only  if $\Phi \in \Delta _2 \cap \nabla _2$ near infinity.

\medskip
 
  The \emph{anisotropic Orlicz-Sobolev class} is defined as
\begin{align}\label{aniso_Wclass}
{W}^{1}_0\mathcal L^\Phi(\Omega)= \{u\text{ - measurable}: & \hbox{ the continuation
of $u$ by $0$ outside $\Omega$} \\ \nonumber  & \hbox{ is weakly differentiable
in $\rn$ and $\nabla u \in \mathcal L^\Phi (\Omega{; \rn})$}\}.
\end{align}
The \emph{anisotropic Orlicz-Sobolev space} ${W}^{1}_0L^\Phi(\Omega)$ is defined accordingly,
on replacing $\mathcal L^\Phi (\Omega{; \rn})$ by~$ L^\Phi (\Omega{; \rn})$ on the right-hand side of equation \eqref{aniso_Wclass}. One has that $W^{1}_0L^\Phi(\Omega)$, equipped with the norm
$$\|u\|_{{W}^{1}_0L^\Phi(\Omega)} = \|\nabla u\|_{L^\Phi (\Omega{; \rn})},$$ is a Banach space.
The Orlicz--Sobolev space ${W}^{1}L^\Phi(\Omega)$ is reflexive, if and only if $\Phi \in \Delta _2 \cap \nabla _2$ near infinity. Then all Orlicz classes coincide, so we abbreviate the notation to
\begin{equation}
\label{W1P}
W^{1,\Phi}(\Omega):=W^1 L^\Phi (\Omega)=W^1 \mathcal{L}^\Phi (\Omega)=W^1 E^\Phi (\Omega).
\end{equation}

\noindent Space $(W_0^{1}L^\Phi(\Omega))'$ is considered endowed with the norm
\[\| H\|_{(W_0^{1}L^\Phi(\Omega))'} =\sup\left\{\frac{\langle H, v\rangle}{\| v\|_{W^{1}L^\Phi(\Omega)}}:\  v\in W_0^{1}L^\Phi(\Omega)\right\},
\]
where $\langle\cdot, \cdot\rangle$ denoted the duality pairing. 
The representation of functionals is given in Lemma~\ref{lem:distr}.

\subsection{Embeddings }\label{ssec:emb}  The statement of optimal anisotropic Sobolev inequality from~\cite{Cfully} requires some further definitions. 
By $\Phi_\circ : [0, \infty ) \to [0, \infty)$ we denote the Young function obeying
\begin{equation}\label{phistar}
|\{\xi \in \R^n: \Phi_\circ (|\xi|) \leq t\}|  =|\{\xi \in \R^n: \Phi
(\xi)\leq t\}|  \quad \hbox{for $t\geq 0$.}
\end{equation}
The function $\R^n \ni \xi \mapsto \Phi _\circ(|\xi|)$ can be regarded as a kind of  \lq\lq average in measure\rq\rq  \, of $\Phi$. 

A basic anisotropic Poincar\'e-type inequality coming from~\cite{barlettacianchi} yields that there exists a~constant $\kappa=c(n)|\Omega|^{-\frac{1}{n}}$ such that
\begin{equation}\label{anisopoinc}
\int_\Omega \Phi_\circ (\kappa|u|)\, dx \leq \int _\Omega \Phi( \nabla u)\, dx
\qquad\text{for every}\quad u
\in W_0^{1}\mathcal{L}^{\Phi} (\Omega ).\end{equation} and \begin{equation}\label{norm-poinc}
 \|u \|_{L^{\Phi _\circ}(\Omega )} \leq c(\kappa)\|\nabla u \|_{L^\Phi (\Omega;\rn )} \qquad\text{for every}\quad u \in W_0^{1}L^{\Phi} (\Omega ).
\end{equation}

We shall pass to Sobolev-type embeddings. When $\Phi _\circ(t)$ is growing slowly close to infinity, with the special case of $t^p$ with $p<n$, we will have $W^{1}L^{\Phi} \subset L^{\Phi_n}$ with $\Phi_n$ prescribed below. If the growth of $\Phi _\circ$ close to infinity is quick, the Orlicz-Sobolev space is embedded into $L^\infty$. For formulating Sobolev inequalities we shall assume control on the values of $\Phi_\circ$ near zero, which for the embedding play no role, see Remark~\ref{rem:convention}. Thus, with no loss of generality, let us start with assuming that
\begin{equation}\label{conv0}
\int _0\bigg(\frac \tau{\Phi _\circ (\tau)}\bigg)^{\frac 1{n-1}}\, d\tau <
\infty\,.
\end{equation}
 
As it was mentioned in Introduction, if $\Phi_\circ$ is growing fast in infinity, Orlicz-Sobolev functions are bounded and continuous and this case is not interesting for us now. Suppose $\Phi_\circ$ is growing so slow in infinity that~\eqref{slow-growth-1} holds, equivalently that
\begin{equation}
\label{intdiv}
\int^\infty\left(\frac{\tau}{\Phi_\circ(\tau)}\right)^\frac{1}{n-1}\,d\tau
=\infty\,,
\end{equation}
and denote by  $H : [0, \infty ) \to [0, \infty)$ is given by
\begin{equation}\label{H1}
    H(t)= \bigg(\int _0^t \bigg(\frac \tau{\Phi _\circ (\tau)}\bigg)^{\frac 1{n-1}}\, d\tau\bigg)^{\frac {n-1}{n}} \quad \hbox{for $\ t \geq 0$,}
\end{equation}
where $H^{-1}$ is the
generalized left-continuous inverse of~$H$. Let $\Phi _n : [0, \infty ) \to [0, \infty]$ be the
Sobolev conjugate of $\Phi$ introduced in \cite{Cfully}. Namely, $\Phi_n$ is the Young function   defined as
\begin{equation}\label{sobconj}
\Phi_n (t)= \Phi _\circ (H^{-1}(t)) \quad \hbox{for $t \geq 0$.}
\end{equation}
By \cite[Theorem 1 and Remark 1]{Cfully}, there exists a constant $\kappa=\kappa(n)$ such that
\begin{equation}\label{B-W} 
\int_{\Omega}\Phi _n\left(\frac{ |u|}{\kappa\,
(\int_{\Omega}\Phi(\nabla u)dy)^{\frac 1{n}}}\right)dx \leq
\int_{\Omega}\Phi(\nabla u)\, dx \quad\text{for every}\quad u \in W_0^{1}\mathcal{L}^{\Phi} (\Omega )
\end{equation}
  and
\begin{equation}\label{B-Wbis}
 \|u \|_{L^{\Phi _n}(\Omega )} \leq \kappa\|\nabla u \|_{L^\Phi (\Omega;\rn )} \quad\text{for every}\quad u \in W_0^{1}L^{\Phi} (\Omega ).
\end{equation}
 Moreover, $L^{\Phi _n}(\Omega )$ is the optimal, i.e. the smallest possible, Orlicz space which renders \eqref{B-Wbis} true for all $n$-dimensional  Young functions $\Phi$ with prescribed $\Phi _\circ$.


\begin{remark}\label{rem:convention}
{\rm Since we are assuming that $|\Omega| < \infty$, inequality 
\eqref{B-Wbis}  continue to hold even if
\eqref{conv0} fails, provided that $\Phi _n$ is defined with
$\Phi_\circ$ replaced by another Young function equivalent near
infinity, which renders \eqref{conv0} true. We shall adopt the convention that $\Phi _n$ is defined according to this procedure in
what follows, whenever needed.}
\end{remark}

 \subsection{Modular convergence and density} The already classical theorems by Gossez~\cite{Gossez} yields density of smooth functions in Orlicz-Sobolev spaces not in norm, but in a weaker topology -- so-called modular one. Due to~\cite{ACCZG,CGWKSG,pgisazg1} this type of result holds also in anisotropic spaces. 
 
  A~sequence $\{U_k\}\subset L^\Phi(\Omega; \R^n)$ is said to converge
modularly to $U$ in $L^\Phi(\Omega{;} \R^n)$ if there exists
$\lambda>0$ such that
\begin{equation}
\label{july41} \lim _{k\to
\infty}\int_{\Omega}\Phi\left(\tfrac{1}{\lambda}(U_k-U)\right)\, dx= 0,
\end{equation} 
or, equivalently, if $U_k\to U$ in measure and there exists
$\lambda>0$ such that $\{\Phi\left(\tfrac{1}{\lambda}U_k\right)\}_k$ is uniformly integrable in $L^1(\Omega)$. 

Since we have~\eqref{B-W}, we say that $u_k\to u$ modularly in $W^{1}L^{\Phi} (\Omega)$ if
\begin{equation}\label{modular-convergence} u_k\to u\quad\text{in }\quad L^{\Phi_n}(\Omega)\qquad\text{and}\qquad \nabla u_k \to
\nabla u \quad \hbox{modularly in $L^\Phi(\Omega;\R^n)$.}
\end{equation}
We will consider Cauchy sequences with respect to this convergence, as well as density of regular functions in topology generated by this convergence. Note that (only) within the doubling regime the modular convergence is equivalent to the norm one. 

\begin{proposition}[Proposition~2.2, \cite{ACCZG}]\label{prop:conv:mod-weak}  Let $\Phi$ be an $n$-dimensional  $N$-function and let $\Omega\subset \rn$ with $|\Omega|<\infty$. Assume that $U_k \to U$ modularly in $L^\Phi(\Omega; \rn)$.  Then there exists a subsequence of $\{U_k\}$, still indexed by $k$, such that
\begin{equation}
\label{july37} \lim _{k \to \infty}\int_\Omega U_k \cdot V\,dx
=  \int_\Omega  U \cdot V \,dx\qquad \text{for every }\quad V \in L^{\widetilde  \Phi}(\Omega;
\rn).
\end{equation}
\end{proposition}

\smallskip

Simple functions are dense in the modular topology 
in anisotropic Orlicz spaces.

\begin{proposition}[Proposition~2.3, \cite{ACCZG}]\label{prop:modulardensity}
 Let $\Phi$ be an $n$-dimensional  $N$-function and let $\Omega$ be a measurable set in $\rn$. Assume that  $U\in L^\Phi(\Omega;\rn)$. Then there exists  a sequence of simple
functions $\{U_k\}$ converging modularly to $U$ in $L^\Phi(\Omega, \rn)$.
\end{proposition}


We present below an anisotropic counterpart of Gossez's approximation theorems, cf.~\cite{Gossez}. 
\begin{proposition}[Proposition~2.4, \cite{ACCZG}]\label{prop:approx-0} Let $\Phi$ be an $n$-dimensional  $N$-function and  let  $\Omega$ be a bounded domain in $\R^n$ having a segment property.  Assume that  $\vp\in W_0^{1}L^{\Phi}(\Omega)\cap L^\infty(\Omega)$. Then there exists  a constant $C=C(\Omega)$ and a sequence $\{\vp_k\} \subset C_0^\infty(\Omega)$ such that $\vp_k \to  \vp$ a.e. in $\Omega$,  $\|\vp_k\|_{L^\infty(\Omega)}\leq C\|\vp\|_{L^\infty(\Omega)}$ {for every $k\in\N$,} and $\vp_k\to \vp$ modularly in $W^{1}L^{\Phi} (\Omega)$.\\
If additionally $\Phi,\wt\Phi\in\Delta_2$, then $\vp_k\to\vp$ in strong (norm) topology in $W^{1}L^{\Phi}  (\Omega)$.
\end{proposition}
The same reasoning based on convolution, but without splitting to segments give the following result, where we do not expect zero trace of approximating sequence.
\begin{proposition}\label{prop:approx} Let $\Phi$ be an $n$-dimensional  $N$-function and  let  $\Omega$ be a bounded  domain in $\R^n$.  Assume that  $\vp\in W^{1}_0L^{\Phi}(\Omega)\cap L^\infty(\Omega)$. Then there exists  a constant $C=C(\Omega)$ and a sequence $\{\vp_k\} \subset C^\infty(\Omega)$ such that $\vp_k \to  \vp$ a.e. in $\Omega$,  $\|\vp_k\|_{L^\infty(\Omega)}\leq C\|\vp\|_{L^\infty(\Omega)}$ {for every $k\in\N$,} and $\vp_k\to \vp$ modularly in $W^{1}L^{\Phi}  (\Omega)$.\\
If additionally $\Phi,\wt\Phi\in \Delta_2$, then $\vp_k\to\vp$ in strong (norm) topology in $W^{1}L^{\Phi}  (\Omega)$.
\end{proposition}

Moreover, we have the following version of compactness.
\begin{proposition}[Anisotropic De La Vall\'ee Poussin Theorem, \cite{CGWKSG}] \label{prop:delaVP}
 Let $\Phi$ be an $n$-dimensional Young function and  $\{U_\sigma\}_{\sigma}$ be a family of measurable vector-valued functions such that $\sup_{\sigma}\int_\Omega \Phi(U_\sigma)\,
dx<\infty$. Then  $\{U_\sigma\}_\sigma$ is uniformly integrable in $L^1$.
\end{proposition}

\subsection{Functionals}
Let us prove the representation of functionals on $W^{1}_0 L^{\Phi}(\Omega)$.
\begin{lemma}\label{lem:distr}If  $\Omega\subset\rn$ is a bounded  domain 
and $H\in (W^{1}_0L^{\Phi}(\Omega))'$, then there exists $\xi\in E^{\wt \Phi}(\Omega;\rn),$ such that\[\langle H, v\rangle=\int_\Omega \xi\cdot \nabla v\,dx\quad\text{for all}\quad v\in {W^{1}_0L^{\Phi}(\Omega)}\]
and $\|H\|_{(W^{1}_0L^{\Phi}(\Omega))'} =\|\xi\|_{ L^{\wt\Phi}(\Omega;\rn)}\,$.
\end{lemma}
\begin{proof}As $\Omega$ is bounded, by Poincar\'e inequality~\eqref{norm-poinc} we can consider ${W^{1}_0L^{\Phi}(\Omega)}$ with gradient norm. Let us note that the map $Tu=\nabla u$ acting $T:{W^{1}_0L^{\Phi}(\Omega)}\to L^\Phi(\Omega;\rn)$ is an isotropy. We set $E=T\big( {W^{1}_0L^{\Phi}(\Omega)}\big)$, equip it with a norm of $L^\Phi(\Omega;\rn)$, and define $S=T^{-1}:E\to W^{1}_0L^{\Phi}(\Omega)$. The map $h\mapsto \langle H, Sh\rangle$ is a continuous linear functional on $E$, so Hahn--Banach theorem (Theorem~\ref{theo:hahn-banach}) and~\eqref{Dual} enable to extend it to a linear functional $\zeta\in E^{\wt \Phi}(\Omega;\rn)$ acting on whole $L^\Phi(\Omega;\rn)$ with $\|\zeta\|_{L^{\wt \Phi}(\Omega;\rn)}=\|H\|_{(W^{1}_0L^{\Phi}(\Omega))'}$. By the Riesz  representation theorem (Theorem~\ref{theo:riesz}) we know that there exists $\xi\in{E^{\wt\Phi}(\Omega;\rn)}$ such that\[\langle \zeta, v\rangle=\int_\Omega \xi\cdot \nabla v\,dx\quad\text{for all }\ v\in C_0^\infty(\Omega).\] 
By Propositions~\ref{prop:conv:mod-weak} and~\ref{prop:modulardensity} the result actually holds for $v\in W^{1}_0L^{\Phi}(\Omega)$. Of course, also  $\|H\|_{(W^{1}_0L^{\Phi}(\Omega))'} =\|\xi\|_{ L^{\wt\Phi}(\Omega;\rn)}$.
\end{proof}

\section{Uniqueness for measure data problems -- Proof of Theorem~\ref{theo:uniq}}\label{sec:uniq}

We  prove uniqueness  of very weak solutions to a broad class anisotropic measure data problem of elliptic type, whose existence is proven in~\cite{ACCZG}.

\begin{proposition}[Theorem~3.4 and 3.10,~\cite{ACCZG}]\label{prop:ex-sola}  $\opA:\Omega\times\rn\to\rn$ satisfies~\eqref{mon} and~\eqref{constants}  with some anisotropic $N$-function $\Phi:\rn\to\rp0$,  $\mu\in \Mb(\Omega)$ and $\Omega$ is a bounded Lipschitz domain in $\rn$, then there exists at least one approximable solution~\eqref{eq:intro}. 
\end{proposition}
In the construction in the original paper it is also shown that solutions to the approximate problems contructed therein satisfy the following property. If  we consider $f_s$ converging to $\mu$ weakly-$\ast$ in the space of measures  such that for every $s>0$ it holds that $\|f_s\|_{L^1(\Omega)}\leq 2\|\mu\|_{\Mb(\Omega)}$ and a
function $u_s$  is a weak solution to $-\dv\opA(x,\nabla u_s)=f_s$
and $ k\to\infty$, then  
%
\begin{equation}
\label{bdd}
\begin{split}\|\nabla T_k(u_s)\|_{L^{\Phi}(\Omega;\rn)}\leq C_0 k\\
 \|\opA(x,\nabla T_k(u_s))\|_{L^{\wt \Phi}(\Omega;\rn)}\leq C_1 k,\\
 \| \opA(\cdot,\nabla T_k(u_s))\cdot\nabla T_k(u_s)\|_{L^1(\Omega)}\leq C_2 k,
\end{split}\end{equation}  where constants $C_0,C_1,C_2>0$ are dependent only on $c_1^\Phi,c_2^\Phi,c_3^\Phi,c_4^\Phi$and $\|\mu\|_{\Mb(\Omega)}$. In fact, in~\cite{ACCZG}  the proof  involves only $c_3^\Phi\in(0,1)$ and $c_1^\Phi,c_2^\Phi,c_4^\Phi=1$ in~\eqref{constants}, but minor modifications enable to justify existence and regularity also in this generality, see \cite{CGWKSG}.  This brings no novelty within the case $\Phi\in\Delta_2\cap\nabla_2$ near infinity, but substantially broadens the scope of investigated problems within the non-doubling regime. 

\medskip

We are in the position to show that if the measure admitts a special form, then the approximable solution $u$ does not depend on the choice of the approximate sequence.

\begin{proof}[Proof of Theorem~\ref{theo:uniq}]
 We suppose $v^1$ and $v^2$ are  solutions obtained as limits of different approximate problems. By assumption there exist sequences of bounded functions $\{f^1_s\}$, $\{ {G^1_s}\}$, $\{f^2_s\}$, $\{G^2_s\}$, such that $f^1_s\to f$ and $f_s^2\to f$ in $L^1(\Omega)$ and $G^1_s\to G$ and $G_s^2\to G$ modularly in $L^{\wt\Phi}(\Omega;\rn)$ and approximate weak solutions $v^j_s$ to~\eqref{eq:intro}, $j=1,2$, such that for a.e. in $\Omega$ we have both $v^1_s\to v^1$ and $v^2_s\to v^2$. Our aim is to prove that  then $v^1=v^2$ a.e. in $\Omega$ even if they are  obtained as limits of different sequences of approximate solutions.

 We fix arbitrary $t,l>0$, use $\phi=T_t(T_lv^1_s-T_lv^2_s)\in W_0^{1}L^{\Phi}(\Omega)\cap L^\infty(\Omega)$  as a test function in~\eqref{prob:trunc} for $v^1$ and $v^2$, and subtract the equations to obtain for every $s>0$\begin{flalign}\nonumber
&\int_{\{|T_lv^1_s-T_lv^2_s|\leq t\}}(\opA(x,\nabla v^1_s)- \opA(x,\nabla v^2_s) )\cdot( \nabla v^1_s-\nabla v^2_s)\,dx\\
&=\int_\Omega (f^1_s-f^2_s)T_t(T_lv^1_s-T_lv^2_s)\,dx+\int_\Omega (G^1_s-G^2_s)\cdot\nabla T_t(T_lv^1_s-T_lv^2_s)\,dx\ =:\, R_s^1+R_s^2.
\label{diff:u-bu}
\end{flalign} The right-hand side above tends to $0$. Indeed, the convergence of $R_s^1$ holds because $|T_t(T_lv^1_s-T_lv^2_s)|\leq t$ and for $s\to 0$ we have $ f^1_s-f^2_s\to 0$ in $L^1(\Omega)$. As for $R_s^2$ it suffices to note that
\begin{flalign*}|R_s^2|&=\left|\int_{\{|T_lv^1_s-T_lv^2_s|\leq t\}}(G^1_s-G^2_s)\cdot\nabla T_lv^1_s\,dx-\int_{\{|T_lv^1_s-T_lv^2_s|\leq t\}}(G^1_s-G^2_s)\cdot\nabla T_lv^2_s\,dx\right|\\
&\leq \left|\int_\Omega(G^1_s-G^2_s)\cdot\nabla T_lv^1_s\,dx\right|+\left|\int_\Omega (G^1_s-G^2_s)\cdot\nabla T_lv^2_s \,dx\right|,
\end{flalign*}  recall the modular convergence of $(G^1_s-G^2_s)\to 0$ in $L^{\wt\Phi}(\Omega;\rn)$, boundedness of the sequence $\{\nabla T_lv^j_s\}_s$ $(j=1,2)$ in $L^{\Phi}(\Omega;\rn)$, and Proposition~\ref{prop:conv:mod-weak}. The left-hand side of~\eqref{diff:u-bu} is nonnegative by monotonicity of $\opA$. Furthermore, by \eqref{bdd} and Fatou's Lemma as $R_s^1+R_s^2\to 0$, we get
\[0\leq
\int_{\{|T_lv^1-T_lv^2 |\leq t\}}(\opA(x,\nabla v^1)-\opA(x,\nabla v^2 ))\cdot( \nabla v^1 -\nabla v^2 )\,dx\leq 0.\]
Consequently,  $\nabla v^1 =\nabla v^2$ a.e. in $\{|T_l v^1 -T_l v^2 |\leq t\}$ for every $t,l>0$, and so for every $k>0$\begin{equation*}
\nabla T_k v^1 =\nabla T_k v^2\quad\text{ a.e. in }\Omega. 
\end{equation*} 
Given the boundary value is the same also $v^1= v^2$ a.e. in $\Omega$. 
\end{proof}

\section{Fully anisotropic capacities  and capacitary measure characterization}\label{sec:cap} Capacities are typically used to describe fine properties of Sobolev functions.   We present the generalization of classical notions of capacities, cf.~\cite{AH,hekima,Ma,Re,Zie}, to anisotropic doubling setting, that is when
\[\Phi\in\Delta_2\cap\nabla_2\quad\text{ near infinity.}\]Mind that due to~\eqref{W1P} in this section we denote an Orlicz--Sobolev space by $W^{1,\Phi}(\Omega)$.

 Let us refer to similar studies of~\cite{MSZ} developed within isotropic Orlicz spaces, but also to \cite{ks} in the variable exponent setting, \cite{Tuo} for related study on metric spaces and~\cite{bahaha} for isotropic considerations in inhomogeneous Musielak--Orlicz spaces.

\medskip

Capacities will be defined by the means of the functional\begin{equation}
\label{Phi-func}
W^{1,\Phi}(\Omega)\ni\vp\mapsto\FPh[\vp]=
\int_\rn \Phi_\circ(\kappa |\vp|)+\Phi(\nabla \vp)\,dx
\end{equation}
with $\Phi_\circ$ from~\eqref{phistar} and $\kappa=c(n,|\Omega|)>0$ from Poincar\'e inequality~\eqref{anisopoinc}. 

\subsection{Sobolev capacity}\label{ssec:sob-cap}
For a set $E\subset\rn$ we define the set of test functions
\[S_\Phi(E)=\left\{ \vp\in W^{1,\Phi}(\rn): \ {\rm int}\, E\subset\{\vp> 1\}, \ \vp\geq 0 \text{ on } \rn\right\}\]
and its Sobolev anisotropic capacity by
\[\capP (E)=\inf\left\{\FPh[\vp]
:\ \vp\in S_\Phi(E)\right\}. \] 
This definition ensures customary properties of capacity. Before their proof let us state the following lemma.
\begin{lemma}\label{lem:grad-min}  If $u, v\in W^{1,\Phi} (\Omega)$, then $\max\{u, v\}$ and
$\min\{u, v\}$ are in $W^{1,\Phi} (\Omega)$ with\[
\nabla\max\{u, v\}(x) =\begin{cases}
\nabla u(x) \text{ a.e. in } \{u\geq v\},\\
\nabla v(x) \text{ a.e. in } \{v \geq u\},\end{cases} 
\nabla \min\{u, v\}(x) =\begin{cases}
\nabla u(x) \text{ a.e. in } \{u \leq v\},\\
\nabla v(x) \text{ a.e. in } \{v \leq u\}.\end{cases}\]
In particular, $|u|\in W^{1,\Phi} (\Omega)$.
\end{lemma} 
\begin{proof}
The proof is exactly like in the isotropic case~\cite{MSZ}.
\end{proof}

\begin{lemma}\label{lem:basic-1}
Sobolev capacity $\capP$ has the following basic properties.
\begin{itemize}
\item[(i)] $\capP(\emptyset)=0.$
\item[(ii)] If $E_1\subset E_2$, then $\capP(E_1)\leq\capP(E_2).$
\item[(iii)] If $E_1,E_2\subset\rn$, then $\capP(E_1\cup E_2)+\capP(E_1\cap E_2)\leq\capP(E_1)+\capP(E_2)$.
\item[(iv)] If $E_1\subset E_2\subset E_3\subset\dots\subset\rn$ are arbitrary sets, then  $\lim_{i\to\infty} \capP(E_i)= \capP(\cup_{i=1}^{\infty}E_i);$
\item[(v)] If $K_1\supset K_2\supset K_3\supset\dots$ are compact sets, then  $\lim_{i\to\infty} \capP(K_i)= \capP(\cap_{i=1}^{\infty}K_i);$
\item[(vi)]  For a countable family of sets $\{E_i\}_i\subset\rn$ we have $\capP(\cup_{i=1}^\infty E_i)\leq\sum_{i=1}^\infty\capP(E_i)$.
\end{itemize}
\end{lemma}
\begin{proof} Properties (i) and (ii)  result directly from the definition of $\capP$. As for (iii) we consider $\Omega=E_1\cup E_2$, $0\leq u,v\in W^{1,\Phi}(\Omega)$, such that $E_1\subset{\rm int }\{u> 1\}$ and $E_2\subset{\rm int}\{v>1\}$. Recall that Lemma~\ref{lem:grad-min} ensures then that $\max\{u,v\},\min\{u,v\}\in W^{1,\Phi}(\Omega)$. Let $A=\{u\geq v\}$ and $B=\{u<v\}$. We observe that
\[\FPh [\max\{u,v\}]=\FPh [u\mathds{1}_A]+\FPh [v\mathds{1}_B]\qquad\text{and}\qquad\FPh [\min\{u,v\}]=\FPh [u\mathds{1}_B]+\FPh [v\mathds{1}_A].\]
Then also\begin{equation}
\label{FPh-sum}
\FPh [\max\{u,v\}]+\FPh [\min\{u,v\}]=\FPh[u]+\FPh [v]
\end{equation}
and (iii) follows.

Let us pass to (iv). Since implication $\lim_{i\to\infty} \capP(E_i)\leq \capP(\cup_{i=1}^{\infty}E_i)$ is straightforward, we just concentrate on the reverse one. We choose $0\leq u_i\in S_\Phi(E_i)$ such that  $\FPh [u_i]\leq \capP(E_i)+\tfrac{\ve}{2^i}$. Let $v_j=\max\{u_1,\dots,u_j\}$. Then $v_j=\max\{v_{j-1},u_j\}$ and $\min\{v_{j-1},u_j\}=1$ inside $E_{j-1}$. Due to~\eqref{FPh-sum} we get
\begin{flalign*}
\FPh [v_j]+\capP(E_{j-1})&\leq 
\FPh [\max\{v_{j-1},u_j\}]+\FPh [\min\{v_{j-1},u_j\}]=\FPh [v_{j-1}]+\FPh [u_j]\\
&\leq \FPh [v_{j-1}] +\capP(E_{j})+\tfrac{\ve}{2^i}.
\end{flalign*}
Iterating the argument we obtain that for every $j$
\[\FPh[v_j]\leq\capP(E_j)+\sum_{i=1}^j\tfrac{\ve}{2^i}<\lim_{i\to\infty}\capP(E_i)+\ve.\] Without loss of generality we may assume that $\lim_{i\to\infty}\capP(E_i)<\infty$ and, consequently, $v_j\in W^{1,\Phi}(\Omega)$. Due to Propositions~\ref{prop:approx} and~\ref{prop:delaVP} there is a subsequence (still called) $\{v_j\}_j$ converging locally weakly in $W^{1,1}(\Omega)$ and strongly in $L^1(\Omega)$ to a certain $v\in W^{1,\Phi}(\Omega)$. By convexity of $\Phi$ we have lower-semicontinuity of $\FPh$, so we get further that
\[\FPh [v]\leq
\liminf_{j\to\infty}\FPh [v_j]\leq\lim_{i\to\infty}\capP(E_i)+\ve.\]
When we define $w=\lim_{j\to\infty}v_j$, then $\cup_{i=1}^\infty E_i\subset{\rm int}\{w=1\}.$ By the uniqueness of the weak limit we infer that $v=w$. Moreover, $v=1$ on $\cup_{i=1}^\infty E_i$ and\[ \capP(\cup_{i=1}^\infty E_i)\leq\FPh [v]\leq \lim_{i\to\infty}\capP(E_i)+\ve.\]

Property (v) follows from the definition as well. In fact it suffices to consider an open set $U\supset\cap_{i=1}^{\infty}K_i$. Then by (ii) we have\[\capP(K)\leq \capP(K_i)\leq\capP(U)\]
and we conclude by taking and infimum over all such $U$.

To prove (vi) we use (iii) and inductional argument to arrive at inequality for any finite family $\{E_i\}_{i=1}^k$, that is $\capP(\cup_{i=1}^k E_i)\leq\sum_{i=1}^k\capP(E_i).$ Then by (iv) we have
\[\capP(\cup_{i=1}^k E_i)\leq\capP(\cup_{i=1}^k E_i)\leq\sum_{i=1}^k\capP(E_i)\leq\sum_{i=1}^k\capP(E_i).\]
\end{proof}
\begin{lemma}\label{lem:cap-open}For any $A\subset\rn$ one has
\[\capP(A)=\inf\{\capP(U):\ \ U\subset\rn\ \text{ is open and }\ A\subset U\}.\]
\end{lemma}
\begin{proof}
We fix $A\subset\rn$ and denote an increasing sequence of sets $\{A_i\}\subset\rn$ such that $A=\cup_{i=1}^\infty A_i$. Let
\[s:=\lim_{i\to\infty}\capP(A_i)\leq \capP(A)\in[0,\infty].\]
To prove the converse inequality we may assume $s<\infty$. For every $i$ we take $u_i\in S_\Phi(A_i)$ such that 
\[\FPh[u_i]\leq \capP(A_i)+\tfrac 1i.\]
Then, since $\Phi$ is doubling, $\{u_i\}$ is a bounded sequence in a reflexive Banach space and, hence, it has a weakly converging subsequence. Let $u$ be its weak limit. By Mazur's Lemma  (Theorem~\ref{theo:mazur}) there exists a sequence $\{v_j\}\subset W^{1,\Phi}(\rn)$ of finite convex combinations of $u_i,$ $i\geq j$ converging strongly to $u$. By convexity of $\FPh$ get that
\[\FPh[v_j]\leq\inf_{i\geq j} \FPh[u_i]\leq s+\tfrac 1i.\] 
Since $u_i\geq 1$ a.e. in some open $U_i\supset A_i$, we construct $V_j$ where $v_j\geq 1$. In fact, we can consider a finite intersection of $U_i$, $i\geq j$. We take a (non-relabelled) subsequence of $v_i$, such that
\[\|v_{j+1}-v_j\|_{W^{1,\Phi}}\leq \tfrac{1}{2^j}.\]
Let us define
\[w_j=v_j+\sum_{i=j}^\infty |v_{i+1}-v_i|\geq v_j+\sum_{i=j}^{k-1} (v_{i+1}-v_i)=v_k\quad\text{for }\ k\geq j\]
and notice that $\{w_j\}\subset W^{1,\Phi}(\rn),$ $w_j\geq 1$ in an opec set $W_j=\cup_{i=j}^\infty V_i\supset A$ and $w_j\to u$ in $W^{1,\Phi}(\rn)$. Since also $v_j\to u$ in $W^{1,\Phi}(\rn)$ we have
\[\capP(A)\leq \lim_{j\to\infty} \FPh[w_j]=\FPh[u]=\lim_{j\to\infty} \FPh[v_j]=\lim_{i\to\infty}\capP(A_i).\]
\end{proof}
By the analogical contruction we can justify with the following conclusion.
\begin{lemma}[Choquet's property]\label{lem:choquet} For an arbitrary set $E\subset\rn$ we have\[\capP(E)=\sup\{\capP(K):\quad K\text{ is compact in } E\}.\] 
\end{lemma}
  
To conclude Corollary~\ref{coro:decomp} we need the following decomposition lemma. Its proof is essentially the one of \cite[Lemma~2.1]{FST} with minor modifications to our notation, but we enclose it for the sake of completeness.
\begin{lemma}\label{lem:basic-decomp}
Suppose $\Omega$ is a bounded set in $\rn$. Then for every $\mu\in\Mb$ there exist unique decomposition $\mu=\mu_0+\mu_1$, such that \begin{itemize}
\item[(a)] $\mu_0(D)=0$ for every measurable set $D$ with $\capP(D)=0$,
\item[(b)] $\mu_1=\mu\mathds{1}_{N}$ for some measurable set $N$ with $\capP(N)=0$,
\end{itemize}
\end{lemma}
\begin{proof} Let $\B$ be a family of measurable subsets of $\Omega$. We shall construct our decomposition making use of an arbitrary fixed sequence of sets $D_1\subset D_2\subset\dots \subset\B$ of sets with $\capP(D_i)=0$, such that  
\[\lim_{i\to\infty}\mu(D_i)=\alpha:=\sup\{\mu(D):\ D\in\B\ \text{ and }\ \capP(D)=0\}<\infty.\]Set $D_\infty=\bigcup_{i=1}^\infty D_i$ and notice that $D_\infty\in\B$, $\capP(D_\infty)=0$ and $\mu(D_\infty)=\alpha$. Then $\mu(D\setminus D_\infty)=0$ for every $D\in\B$ with $\capP(D)=0$. By defining $\mu_0=\mu\mathds{1}_{\rn\setminus D_\infty}$ and $\mu_1=\mathds{1}_{D_\infty}\mu$ we get the decomposition of the desired properties. In particular, uniqueness of the decomposition is evident.
\end{proof}

\noindent\textbf{Lebesgue's points.} We shall show  that sets of  $\Phi$-capacity zero are removable for Orlicz-Sobolev functions.

\begin{definition}\label{def:quasi}
Function $u$ is called {\em $\Phi$-quasicontinuous} if for
every $\e>0$ there exists an open set $U$ with $\capP (U)<\e$, such that $f$ restricted to $\Omega\setminus U$ is continuous. We say that a~claim holds {\em $\Phi$-quasieverywhere} if it holds outside a set of Sobolev $\Phi$-capacity zero.
\end{definition}

 Let us start with the following observation. 

\begin{lemma}\label{lem:almost-uniform} For each Cauchy
sequence with respect to the $W^{1,\Phi} (\Omega)$-modular of functions from $C(\rn)\cap W^{1,\Phi} (\Omega)$ there is a subsequence which converges pointwise $\Phi$-quasieverywhere in $\Omega$. Moreover, the convergence is uniform outside a set of arbitrary small Sobolev $\Phi$-capacity.
\end{lemma}
\begin{proof} Let $\{u_i\}_i\subset C(\rn)\cap W^{1,\Phi} (\Omega)$ be a Cauchy sequence with respect to the $W^{1,\Phi} (\Omega)$--modular topology. We can fix $\lambda>0$ and extract a subsequence such that \[\int_\Omega \Phi\left(\frac{2^i}{\lambda}(\nabla u_i-\nabla u_{i+1})\right)\,dx\leq \frac{1}{4^i},\quad i\in\N.\] We denote 
\[W_i:=\{x\in\rn:\ 2^i|u_i-u_{i+1}|>\lambda\},  \  i\in\N,\quad \text{and}\quad Z_j:=\bigcup_{i=j}^\infty W_i,  \  j\in\N.\]
Due to Lemma~\ref{lem:grad-min}, $v_i:= 2^i|u_i-u_{i+1}|\in W^{1,\Phi}(\Omega)$. In the view of~\eqref{anisopoinc} notice that 
\[\capP(W_i)\leq c\int_\Omega \Phi\left(\tfrac{1}{\lambda}\nabla v\right)dx= c\int_\Omega \Phi\left(\tfrac{2^i}{\lambda}(\nabla u_i-\nabla u_{i+1})\right)dx\leq c\,2^{-i}.\]
Subadditivity of Sobolev anisotropic capacity implies that
\[\capP(Z_j)\leq \sum_{i=j}^\infty \capP(W_i)\leq c\sum_{i=j}^\infty 2^{-i}\leq c\,2^{1-j}.\]
Thus, $\capP\left(\bigcap_{j=1}^\infty Z_j\right)\leq \lim_{j\to\infty}\capP(Z_j)=0.$ We have proven that $\{u_i\}$ converges pointwise in $\rn\setminus \bigcap_{j=1}^\infty Z_j$, that is except for a set of Sobolev anisotropic capacity zero.

As for the second claim, it suffices to realize that 
\[|u_l(x)-u_k(x)|\leq \sum _{i=l}^{k-1} |u_i(x)-u_{i+1}(x)|\leq c \sum _{i=l}^{k-1} 2^{-i}<c\,2^{1-l}\]
for every $x\in \rn\setminus Z_j$ and every $k>l>j$. Therefore, the convergence is uniform in $\rn\setminus Z_j$.
\end{proof}

\begin{proposition}\label{prop:quasicont}
For each $u\in W^{1,\Phi} (\Omega)$ there exists a unique $\Phi$--quasicontinuous function $v \in W^{1,\Phi} (\Omega)$  such that $u=v$
almost everywhere in $\Omega$.
\end{proposition}
\begin{proof} Let $u\in W^{1,\Phi} (\Omega)$. Due to Proposition~\ref{prop:approx}, it follows that there exists a sequence of functions $\{u_i\}_i\subset C^\infty(\Omega)\cap W^{1,\Phi}(\Omega)$, such that $u_i\to u$ modularly in $W^{1,\Phi} (\Omega)$. By
Lemma~\ref{lem:almost-uniform} there exists a subsequence that converges uniformly outside a set of arbitrarily small capacity. Note that the uniform convergence implies continuity of the limit function, so we get that $u$  restricted to complement of a set of arbitrarily small capacity is continuous, what was to be shown.
\end{proof}
\subsection{Relative capacity}\label{ssec:rel-cap}
We shall consider also  anisotropic relative capacity (that can be called also  anisotropic variational capacity). With this aim, for every $K$ compact in $\Omega$ this let us denote\begin{equation}
\label{def:R}
\mathcal{R}_\Phi(K,\Omega) := \{u\in W^{1,\Phi} (\Omega)\cap C_0(\Omega):\quad u\geq 1\ \text{ on $K$ and }\ u \geq 0\}.
\end{equation}
We recall $\FPh$ from~\eqref{Phi-func} and we set
\[\CapP(K,\Omega) :=\inf\left\{ 
\FPh[\vp]:\ \ \vp\in \mathcal{R}_\Phi(K,\Omega )\right\} .\]
For open  sets $A\subset\Omega$ we define\[\CapP (A,\Omega)=\sup\left\{\CapP(K,\Omega):\quad K\subset A \text{ and } K \text{ is compact in }A\right\} \]
and to an arbitrary set $E\subset\Omega$ by\[\CapP (E,\Omega)=\inf\left\{\CapP(A,\Omega):\quad E\subset A \text{ and } A \text{ is open in }\Omega\right\}. \]
By the same arguments as in Lemmas~\ref{lem:cap-open} and~\ref{lem:choquet} this notion of capacity enjoys Choquet property, i.e.
\[\CapP(E,\Omega)=\sup\{\CapP(K,\Omega):\quad K\text{ is compact in } E\}.\] 
When a set $K$ is compact, due to Proposition~\ref{prop:approx} each function $u\in \mathcal{R}_\Phi(K,\Omega)$ can be modularly approximated by smooth functions. Consequently, 
\[\CapP(K,\Omega) =\inf\left\{
\FPh[\vp]:\ \ \vp\in C^\infty(\Omega )\cap 
\mathcal{R}_\Phi(K,\Omega) \right\}. \]  

\begin{remark}\label{rem:cap-cap}
Obviously, for a set $E\subset\Omega\subset\rn$ we have $\capP(E)\leq\CapP(E,\Omega)$.
\end{remark}

\subsection{Sets of zero capacity.} 
Directly from definition we have the following properties.
\begin{lemma} Each set of  anisotropic capacity zero is contained in a Borel set of  anisotropic capacity zero. Countable union of sets of  anisotropic capacity zero has  anisotropic capacity zero.\end{lemma}
 Moreover, having Poincar\'e inequality~\eqref{anisopoinc} we infer the following fact.
\begin{lemma} If $B_R$ is a ball in $\rn$, $E\subset B_R$ and $\CapP(E,B_R)=0$, then $|E|=0$.\\
If a  set $E\subset\Omega\subset\rn$ satisfies $\CapP(E,\Omega)=0$ also $\capP(E)=0$.
\end{lemma}
\noindent On the other hand, it is possible that a set has measure zero, but positive capacity.

\subsection{Approximation}

\begin{lemma}\label{lem:dual-approx}
Suppose $\nu\in (W_0^{1,\Phi}(\Omega))'\cap\Mb(\Omega)$. For every $\ve>0$ there exists $f\in C_0^\infty(\Omega)$, such that
\[\|f-\nu\|_{(W_0^{1,\Phi}(\Omega))'}\leq\ve\quad\text{and}\quad \|f\|_{L^1(\Omega)}\leq|\nu|(\Omega).\]
\end{lemma}
\begin{proof}There is nothing to prove if $|\nu|(\Omega)=0$, so without loss of generality we assume the converse. Let us define
\[\mathcal{E}:=\{f\in C_0^\infty(\Omega):\ \|f\|_{L^1(\Omega)}\leq|\nu|(\Omega)\},\]
which is a closed and convex set. Suppose, by contradiction, that there exists $\ve>0$, such that no $f\in\mathcal{E}$ satisfies $\|f-\nu\|_{(W_0^{1,\Phi}(\Omega))'}\leq\ve$. Then $\nu$ does not belong to the closure of $\mathcal{E}$ in $(W_0^{1,\Phi}(\Omega))'$. Hyperplane separation theorem (Theorem~\ref{theo:hyp-sep}) ensures that a point $\nu$ can be separated from a convex set $\mathcal{E}$, namely that there exist $u\in W_0^{1,\Phi}(\Omega)$ and $s\in\R$ such that
\begin{equation}
\label{separ}\sup_{f\in\mathcal{E}}\int_\Omega f u\,dx\leq s<\int_\Omega  u\,d\nu.
\end{equation} The first inequality in the last display implies that $u\in L^\infty(\Omega)$ and \[\|u\|_{L^\infty(\Omega)}\leq\frac{s}{|\nu|(\Omega)}.\] Since $u\in W_0^{1,\Phi}(\Omega)\cap L^\infty(\Omega)$, also $|u|\leq \|u\|_{L^\infty(\Omega)}$ $\capP$-quasi-everywhere, hence also $|\nu|$-a.e. in $\Omega$. Therefore
\[\int_\Omega u\,d\nu\leq \|u\|_{L^\infty(\Omega)}|\nu|(\Omega)\leq\frac{s}{|\nu|(\Omega)} |\nu|(\Omega)\leq s,\]
but this contradicts with the second inequality of~\eqref{separ}.
\end{proof}

\subsection{Measure characterization -- Proof of Theorem~\ref{theo:decomp}}\label{ssec:decomp}
Before we give the proof of Theorem~\ref{theo:decomp} we shall   concentrate on the absolute continuity of $\mu\in L^1(\Omega)+(W_0^{1,\Phi}(\Omega))'$  with respect to $W^{1,\Phi}$-capacity. Notice that for a~nonnegative measure $\mu$, such that $\mu =f+\dv G \in L^1(\Omega)+(W_0^{1,\Phi}(\Omega))'$ and an arbitrary set $E\subset\Omega$ we have for every ${\vp\in W^{1,\Phi}_0(\Omega)}$
\[\mu(E)\leq \int_E f\,\vp\,dx+\int_E G\cdot\nabla\vp\,dx\leq \|f\|_{L^1(E)}\|\vp\|_{L^\infty(R)}+\|G\|_{L^{\wt\Phi}(\Omega;\rn)}\|\nabla \vp\|_{L^{\Phi}(\Omega;\rn)}.\]  

\begin{lemma}\label{lem:cap0}
If $\Omega\subset\rn$, $\mu\in L^1(\Omega)+(W_0^{1,\Phi}(\Omega))'$ and a set $E\subset\Omega$ is such that $\CapP(E,\Omega)=0$, then $\mu(E)=0$.
\end{lemma}
\begin{proof}By the assumption there exist $f\in L^1(\Omega)$ and $G\in L^{\wt\Phi}(\Omega;\rn)$, such that $\mu=f-\dv G$ in the sense of distributions. Then obviously $\mu\in\Mb(\Omega)$. Moreover, there exist a Borel set $E_0\supset E$ with $\CapP(E_0,\Omega)=0$. We fix compact $K\subset E_0$ and open $\Omega'\subset\Omega$, such that $K\subset\Omega'$. Obviously $\CapP(K,\Omega)=0$. Let us consider a sequence $\{\vp_j\}_j\subset C^\infty(\Omega')\cap \mathcal{R}_\Phi(K,\Omega')$ of functions such that \[
\int_{\Omega'}\Phi(\tfrac{1}{\lambda} \nabla\vp_j)\,dx\xrightarrow[j\to\infty]{}0\ \text{ for some $\lambda>0$},\]
which is possible due to Proposition~\ref{prop:approx-0}. Then
\[|\mu(K)|\leq \left|\int_{\Omega'}\vp_j\,d\mu\right|=\left|\int_{\Omega'}f\,\vp_j\,dx + \int_{\Omega'}G\cdot\nabla \vp_j\,dx\right|.\]
We take infimum with respect to $j$, make use of Proposition~\ref{prop:conv:mod-weak} and obtain that
\[|\mu(K)|\leq C\CapP(K,\Omega')\qquad\text{with}\qquad C=C(\|f\|_{L^1(\Omega')},\|G\|_{L^{\wt \Phi}(\Omega')}).\] This implies 
\[\mu(E)\leq\mu(E_0)=\sup\{\mu(K):\ K\subset E_0,\ K \text{ compact}\}=0.\] 
\end{proof}

To get the measure characterization we follow basic ideas from~\cite{BGO} with classical growth later used in~\cite{IC-measure-data,Zhang}. 

\begin{proof}[Proof of Theorem~\ref{theo:decomp}]

The implication: if $\mu$ belongs to $L^1(\Omega) + (W_0^{1,\Phi}(\Omega))'$, then $\mu\in M^\Phi_b(\Omega)$ is provided in Lemma~\ref{lem:cap0}. We shall concentrate now on the converse, that is, if $\mu_\Phi\in M^\Phi_b(\Omega)$, then $\mu_\Phi\in L^1(\Omega) + (W_0^{1,\Phi}(\Omega))'$.

\medskip

\noindent {\em Step 1. Initial decomposition. }\\ The aim of this step is to show that for a nonnegative $\mup\in \MP(\Omega)$ we can find a positive measure $\gam\in(W^{1,\Phi}_0(\Omega))'$  and positive Borel measurable function $h$ belonging to~weighted Lebegue's space $L^1(\Omega,\gam)$ such that $d\mup=h\,d\gam$.

Due to Proposition~\ref{prop:quasicont} for any $\wt u\in W_0^{1,\Phi}(\Omega)$ we can find its uniquely defined $\Phi$-quasi-continuous representative denoted by $u$. We define a functional $F: W_0^{1,\Phi}(\Omega)\to [0,\infty]$ by
\[F[u]=\int_\Omega  u_+\,d\mup\]
and observe that it is convex and lower-semicontinuous on a space $W_0^{1,\Phi}(\Omega)$ which is separable since $\Phi\in\Delta_2\cap\nabla_2$. Thus, $F$ can be expressed as a supremum of a countable family of continuous affine functions. In fact, there exist sequences of functions $\{\xi_n\}_n\subset (W_0^{1,\Phi}(\Omega))'$ and numbers $\{a_n\}_n\subset\rn$ such that
\[F[u]=\sup_{n\in\N}\{\langle\xi_n,u\rangle-a_n\}\quad\text{for all }\ u\in W_0^{1,\Phi}(\Omega).\]
Then, for any $s>0$, $sF[u]=F[su]\geq s\langle\xi_n,u\rangle-a_n$ for every $n$. By dividing by $s$ and letting $s\to\infty$ we get $F[u]\geq \langle\xi_n,u\rangle$ for all $u\in W_0^{1,\Phi}(\Omega)$. Since $F[0]= 0$, it follows that $a_n\geq 0$. Thus $F[u]\geq \sup_{n\in\N} \langle\xi_n,u\rangle\geq \sup_{n\in\N}\{\langle\xi_n,u\rangle-a_n\}=F[u]$
 and, consequently, \begin{equation}
 \label{F} 
 F[u]=\sup_{n\in\N} \langle\xi_n,u\rangle.
 \end{equation} In turn, for all $\vp\in C_0^\infty(\Omega)$ we have
 \[ \langle\xi_n,\vp\rangle\leq  \sup_{n\in\N}\langle\xi_n,\vp\rangle=F[\vp]=\int_\Omega \vp_+\,d\mup\leq \|\mup\|_{\Mb(\Omega)} \|\vp\|_{L^\infty(\Omega)}.\]
 By the same arguments for $-\vp$ we get \[| \langle\xi_n,\vp\rangle|\leq  \|\mup\|_{\Mb(\Omega)} \|\vp\|_{L^\infty(\Omega)}\]
 implying that $\xi_n\in (W_0^{1,\Phi}(\Omega))'\cap \Mb(\Omega).$ By the Riesz representation theorem (Theorem~\ref{theo:riesz}) there exists nonnegative $\xinm\in\Mb(\Omega)$, such that 
 \[\langle\xi_n,\vp\rangle=\int_\Omega \vp \,d\xinm \quad\text{for all }\vp\in C_0^\infty(\Omega).\] 
We observe that \begin{equation}
 \label{xinm-mu}
\xinm\leq\mup\quad\text{and}\quad\|\xinm\|_{\Mb(\Omega)}\leq \|\mup\|_{\Mb(\Omega)}.
 \end{equation}
 Let us define\begin{equation}
 \label{eta}\eta=\sum_{n=1}^\infty\frac{\xi_n}{2^n(\|\xi_n\|_{(W_0^{1,\Phi}(\Omega))'}+1)}
 \end{equation}
 and note that the series in absolutely convergent in $(W_0^{1,\Phi}(\Omega))'$. Hence, whenever $\vp\in C_0^\infty(\Omega)$ we can write
 \begin{flalign*}
|\langle\eta,\vp\rangle|&\leq \sum_{n=1}^\infty\frac{|\langle \xi_n,\vp\rangle|}{2^n(\|\xi_n\|_{(W_0^{1,\Phi}(\Omega))'}+1)}\leq \sum_{n=1}^\infty\frac{\|\xinm\|_{\Mb(\Omega)}}{2^n}\|\vp\|_{L^\infty(\Omega)}\\&\leq \|\mup\|_{\Mb(\Omega)}\|\vp\|_{L^\infty(\Omega)}
\end{flalign*}
and $\eta\in (W_0^{1,\Phi}(\Omega))'\cap \Mb(\Omega)$ too.

Taking now\[\etam=\sum_{n=1}^\infty\frac{\xinm}{2^n(\|\xi_n\|_{(W_0^{1,\Phi}(\Omega))'}+1)}\]
we deal with the series of positive elements that is absolutely convergent in $ \Mb(\Omega)$. Moreover, $\xinm\ll \etam$ and thus for every $n$ there exists a nonnegative function $h_n\in L^1(\Omega,d\etam)$ such that $d\xinm=h_n\, d\etam$ and -- according to~\eqref{F} -- we get that
\begin{equation}
\label{functional}\langle\mup,\vp\rangle=
\int_\Omega \vp\,d\mup=\sup_{n\in\N}
\int_\Omega \vp\,d\xinm=\sup_{n\in\N}\int_\Omega h_n\,\vp\,d\etam\qquad\text{for any }\ \vp\in C_0^\infty(\Omega).
\end{equation} 
On the other hand,~\eqref{xinm-mu} ensures that $h_n\,\etam\leq\mup$, i.e. for any measurable set $E\subset\Omega$ and every $n$ we have \begin{equation}
\int_E h_n\,d \etam\leq \mup(E).
\end{equation} 
We denote $h_{\max}^{k}=\max\{h_1(x),\dots,h_k(x)\}$ and
\begin{equation}
\label{Es}E^{j,k}=\{x\in E:\ \ h_{\max}^{j}(x)>h_i(x)\ \text{ for every }\ i=1,\dots,k-1\}.
\end{equation}
Then $E^{j,k}$ for $j=1,\dots,k$ are pairwise disjoint and $E=\cup_{j=1}^k E^{j,k}$, so\begin{equation}
\int_E h_{\max}^k(x)\,d \etam\leq \sum_{j=1}^k
\int_{E^{j,k}} h_{\max}^k(x)\,d \etam\leq  \sum_{j=1}^k \mup(E^{j,k})=\mup(E).
\end{equation} 
Letting $k\to\infty$ and taking $h(x)=\sup_{j\in\N} h_j(x)$, we get for any measurable set $E\subset\Omega$\begin{equation}
\int_E h\,d \etam\leq \mup(E).
\end{equation} 
Having~\eqref{functional}, for every nonnegative $\vp\in C_0^\infty$ we have
\begin{flalign*}
\int_\Omega \vp\,d\mup=\sup_{j\in\N} \int_\Omega h_j \vp\,d\etam\leq \int_\Omega h \vp\,d\etam\leq \int_\Omega \vp\,d\mup
\end{flalign*}
that is \[d\mup=h\,d\etam.\] Since $\mup(\Omega)\in\Mb(\Omega)$, we deduce that $h\in L^1(\Omega,d\etam)$ and the aim of this step is achieved with $\gam=\etam\in(W_0^{1,\Phi}(\Omega))'$.

\medskip

\noindent {\em Step 2. Auxiliary sequence of measures. }\\
 Let $\{K_i\}_i$ be an increasing sequence of sets compact  in $\Omega$, such that $\cup_{i=1}^\infty K_i=\Omega$. We denote \[\wt\mu_i=T_i(h\mathds{1}_{K_i})\gam\qquad\text{for every }\ i\in\N\] and notice that $\{\wt\mu_i\}_i$ is an increasing sequence of positive measures in $(W^{1,\Phi}_0(\Omega))'$ with  supports compact in $\Omega$.  Set \begin{equation} \label{mui} \mu_0 =\wt\mu_0\equiv 0\qquad\text{ and }\qquad \mu_i=\wt\mu_i-\wt\mu_{i-1}\quad\text{for every }\ i\in\N. \end{equation} Then $\sum_{i=1}^m\mu_i=T_m(h\mathds{1}_{K_m})\gam\in(W^{1,\Phi}_0(\Omega))'\cap\Mb(\Omega)$. Since $\mu_i\geq 0$, we have also that $\sum_{i=1}^\infty \|\mu_i\|_{\Mb(\Omega)}<\infty$. Furthermore, $\mup=\sum_{i=1}^\infty\mu_i$ and this series is absolutely convergent in $\Mb(\Omega)$. 

\medskip

\noindent {\em Step 3. Construction of decomposition. }\\
By Lemma~\ref{lem:dual-approx} for fixed $i$ we can find $k_i$ large enough, so that \[ f_i=\mu_{i,k_i}\in C_0^\infty(\Omega)\qquad \text{and}\qquad g_i=\mu_i-\mu_{i,k_i}\in (W_0^{1,\Phi}(\Omega))'\cap \Mb(\Omega),\]
such that \begin{equation}
\label{male}\|\mu_{i,k_i}\|_{L^1 (\Omega)}\leq \|\mup\|_{\Mb(\Omega)} \qquad \text{and}\qquad \|\mu_{i}-\mu_{i,k_i}\|_{(W_0^{1,\Phi}(\Omega))'}\leq 2^{-k_i}.
\end{equation}
 Up to a subsequence --  the series $\sum_{i=1}^\infty f_i$ is convergent in $L^1 (\Omega)$. We denote its limit as 
$f^0=\sum_{i=1}^\infty f_i\in L^1(\Omega)$. The convergence of $\{g_i\}$ follows directly from~\eqref{male}. We note that the series $\sum_{i=1}^\infty g_i$ converges in $(W_0^{1,\Phi}(\Omega))'$ and there exists its limit $g^0=\sum_{i=1}^\infty g_i\in(W_0^{1,\Phi}(\Omega))'$. Since $\mu_i=f_i+g_i$, the three series $\sum_{i=1}^\infty \mu_i,$ $\sum_{i=1}^\infty f_i,$ and $\sum_{i=1}^\infty g_i$ converge in the sense of distributions and, consequently, $\mup=f^0+g^0$.

\medskip

\noindent {\em Step 4.} Summary. If the measure was not nonnegative we conduct the above reasoning  on decomposition $\mup=(\mup)_++(\mup)_-$ separately for its positive and negative part.  Note that by monotonicity of capacity if $\capP(A)=0$, then $(\mup)_+(A)=0=(\mup)_-(A)$ and the  anisotropic capacity can be achieved over Borel sets included in $A$, see Lemma~\ref{lem:basic-1}. Clearly, wherever $\mu_\Phi$ is positive, so is $f$.  Thus for a signed measure $\mu\in\Mb(\Omega)$ it is equivalent that $\mu\in \MP(\Omega)$ and $\mu \in \Mb(\Omega)\cap\big(L^1(\Omega)+(W_0^{1,\Phi}(\Omega))'\big)$, that is -- by Lemma~\ref{lem:distr} -- when there exists $f\in L^1(\Omega)$ and $G\in L^{\wt\Phi}(\Omega;\rn)$, such that 
\[\mup=f-\dv G\ \text{ in the sense of distributions}.\]Indeed,  the proof starts with justification that $\mu \in \big(\Mb(\Omega)\cap(L^1(\Omega)+(W_0^{1,\Phi}(\Omega))')\big)\subset \MP(\Omega) .$ Thus, the proof of the capacitary characterization is completed.
\end{proof}

\section*{Appendix} 

We make use of the following classical results.

\begin{theorem}[Mazur's Lemma]\label{theo:mazur}
If $\{x_n\}$ converges weakly to $x$ in a reflexive Banach space $X$, then there exists a sequence $\{y_n\}_n\subset X$ made up of finite convex combinations of $x_n$'s such that  $y_n\to x$ (strongly) in $X$.
\end{theorem}

\begin{theorem}[Riesz representation]\label{theo:riesz}Let $\mu$ be a Radon measure on bounded $\overline{\Omega}\subset\rn$. Then there is a unique signed Borel measure $\nu$ on $\overline{\Omega}$ (that is, a measure defined on Borel sets of $\overline{\Omega}$) such that
\[\langle\mu, u\rangle=\int_{\overline{\Omega}} u\, d\nu\qquad\text{for every $u\in C(\overline{\Omega})$}.\]
\end{theorem}

\begin{theorem}[Hahn--Banach extension]\label{theo:hahn-banach}
Let $Y$ be a subspace of a real normed linear space $X$, and suppose that $\theta$ is a continuous linear functional on $Y$ for which there exists a positive constant $M$ satisfying $|\theta(y)| \leq M\|y\|_Y$ for all $y\in Y$. Then there exists an extension of $\vp$ to a continuous linear functional $\vartheta$ on $X$, such that $|\vartheta(x)|\leq M\|x\|_X$ for all $x\in X$.
\end{theorem} 
 
\begin{theorem}[Hahn--Banach hyperplane separation] If $A,B$ are non-empty, closed and disjoint convex subsets of a normed space $X$ and $B$ is compact, then there exists a~continuous linear functional $\vartheta$ on $X$ and $s\in\R$, such  that for all $a\in A$ and $b\in B$ we have $\vartheta(a)\leq s<\vartheta(b)$.\label{theo:hyp-sep}
\end{theorem}

\end{document}